\documentclass[11pt]{article}
\usepackage{graphicx}
\usepackage{latexsym}
\usepackage{amsfonts}
\usepackage{mathrsfs,amsthm,amsmath}
\usepackage{color}
\newtheorem{theorem}{Theorem}[section]
\newtheorem{proposition}[theorem]{Proposition}

\newtheorem{lemma}{Lemma}[section]
\newtheorem{remark}{Remark}[section]

\begin{document}
\title{Random time-changes and asymptotic results for a class of
continuous-time Markov chains on integers with alternating
rates\thanks{The authors acknowledge the support of: GNAMPA and GNCS groups
of INdAM (Istituto Nazionale di Alta Matematica); MIUR--PRIN 2017, Project
‘Stochastic Models for Complex Systems’ (no. 2017JFFHSH); MIUR Excellence
Department Project awarded to the Department of Mathematics, University 
of Rome Tor Vergata (CUP E83C18000100006).}}
\author{Luisa Beghin\thanks{Dipartimento di Scienze Statistiche,
Sapienza Universit\`{a} di Roma, Piazzale Aldo Moro 5, 00185 Rome,
Italy. e-mail: \texttt{luisa.beghin@uniroma1.it}}\and Claudio
Macci\thanks{Dipartimento di Matematica, Universit\`a di Roma Tor
Vergata, Via della Ricerca Scientifica, 00133 Rome, Italy. e-mail:
\texttt{macci@mat.uniroma2.it}}\and Barbara
Martinucci\thanks{Dipartimento di Matematica, Universit\`{a} degli
Studi di Salerno, Via Giovanni Paolo II n. 132, 84084 Fisciano,
SA, Italy. e-mail: \texttt{bmartinucci@unisa.it}}}
\date{}
\maketitle
\begin{abstract}
We consider continuous-time Markov chains on integers which allow
transitions to adjacent states only, with alternating rates. We
give explicit formulas for probability generating functions, and
also for means, variances and state probabilities of the random
variables of the process. Moreover we study independent random
time-changes with the inverse of the stable subordinator, the
stable subordinator and the tempered stable subodinator. We also
present some asymptotic results in the fashion of large
deviations. These results give some generalizations of those
presented in \cite{DicrescenzoMacciMartinucci}.
\end{abstract}

\section{Introduction}
We consider a class of continuous-time Markov chains on integers
which can have transitions to adjacent states only, and with
alternating transition rates to their adjacent states; namely we
assume to have the same transition rates for the odd states, and
the same transition rates for the even states. We recall that 
Markov chains with alternating rates are useful in the study of 
chain molecular diffusion; see e.g. \cite{TarabiaTakagiElbaz} 
and other references cited in \cite{DicrescenzoMacciMartinucci}.
In this paper we also study independent random time-changes of 
these Markov chains with the inverse of the stable subordinator
and the (possibly tempered) stable subordinator.

We give a more rigorous presentation in terms of the generator. In
general we consider a continuous-time Markov chain $\{X(t):t\geq
0\}$ on $\mathbb{Z}$ (where $\mathbb{Z}$ is the set of integers),
and we consider the state probabilities
\begin{equation}\label{eq:pmf-notation}
p_{k,n}(t):=P(X(t)=n|X(0)=k),
\end{equation}
which satisfy the condition $p_{k,n}(0)=1_{\{k=n\}}$; the
generator $G=(g_{k,n})_{k,n\in\mathbb{Z}}$ of $\{X(t):t\geq 0\}$
is defined by
$$g_{k,n}:=\lim_{t\to 0}\frac{p_{k,n}(t)-p_{k,n}(0)}{t}.$$
Then, for some $\alpha_1,\alpha_2,\beta_1,\beta_2>0$, we assume to
have (see Figure \ref{fig1})
$$g_{k,n}:=\left\{\begin{array}{ll}
\alpha_1&\ \mbox{if}\ n=k+1\ \mbox{and}\ k\ \mbox{is even}\\
\beta_1&\ \mbox{if}\ n=k+1\ \mbox{and}\ k\ \mbox{is odd}\\
\alpha_2&\ \mbox{if}\ n=k-1\ \mbox{and}\ k\ \mbox{is even}\\
\beta_2&\ \mbox{if}\ n=k-1\ \mbox{and}\ k\ \mbox{is odd}\\
0&\ \mbox{otherwise}
\end{array}\right.\ (\mbox{for}\ k\neq n);$$
therefore
$$g_{n,n}=\left\{\begin{array}{ll}
-(\alpha_1+\alpha_2)&\ \mbox{if}\ n\ \mbox{is even}\\
-(\beta_1+\beta_2)&\ \mbox{if}\ n\ \mbox{is odd}.
\end{array}\right.$$

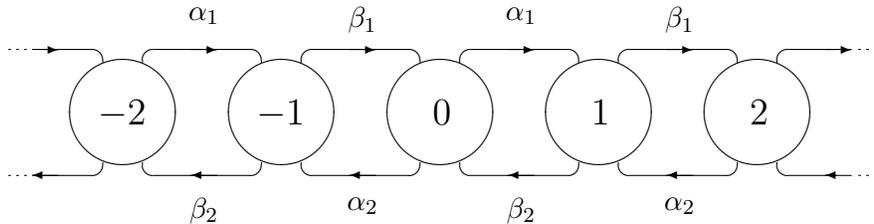
\begin{figure}\label{fig1}
\begin{center}
\begin{picture}(341,91)
\put(39.25,45.25){\circle{38}} \put(99.25,45.25){\circle{38}}
\put(159.25,45.25){\circle{38}} \put(219.25,45.25){\circle{38}}
\put(279.25,45.25){\circle{38}}
\put(9.25,64){\oval(45,10)[rt]} \put(69.25,64){\oval(45,10)[t]}
\put(129.25,64){\oval(45,10)[t]} \put(189.25,64){\oval(45,10)[t]}
\put(249.25,64){\oval(45,10)[t]} \put(309.25,64){\oval(45,10)[lt]}
\put(9.25,26.3){\oval(45,10)[rb]}
\put(69.25,26.3){\oval(45,10)[b]}
\put(129.25,26.3){\oval(45,10)[b]}
\put(189.25,26.3){\oval(45,10)[b]}
\put(249.25,26.3){\oval(45,10)[b]}
\put(309.25,26.3){\oval(45,10)[lb]}
\put(19,35){\makebox(40,15)[t]{\Large $-2$}}
\put(79,35){\makebox(40,15)[t]{\Large $-1$}}
\put(140,35){\makebox(40,15)[t]{\Large $0$}}
\put(200,35){\makebox(40,15)[t]{\Large $1$}}
\put(260,35){\makebox(40,15)[t]{\Large $2$}}
\put(5,69){\vector(1,0){10}} \put(65,69){\vector(1,0){10}}
\put(125,69){\vector(1,0){10}} \put(185,69){\vector(1,0){10}}
\put(245,69){\vector(1,0){10}} \put(305,69){\vector(1,0){10}}
\put(15,21.3){\vector(-1,0){10}} \put(75,21.3){\vector(-1,0){10}}
\put(135,21.3){\vector(-1,0){10}}
\put(195,21.3){\vector(-1,0){10}}
\put(255,21.3){\vector(-1,0){10}}
\put(315,21.3){\vector(-1,0){10}}
\put(35,70){\makebox(70,15)[t]{$\alpha_1$}}
\put(105,70){\makebox(50,15)[t]{$\beta_1$}}
\put(155,70){\makebox(70,15)[t]{$\alpha_1$}}
\put(225,70){\makebox(50,15)[t]{$\beta_1$}}
\put(45,-2){\makebox(50,15)[t]{$\beta_2$}}
\put(95,-2){\makebox(70,15)[t]{$\alpha_2$}}
\put(165,-2){\makebox(50,15)[t]{$\beta_2$}}
\put(215,-2){\makebox(70,15)[t]{$\alpha_2$}}
\put(3,69){\line(-1,0){1}} \put(3,21.3){\line(-1,0){1}}
\put(0,69){\line(-1,0){1}} \put(0,21.3){\line(-1,0){1}}
\put(-3,69){\line(-1,0){1}} \put(-3,21.3){\line(-1,0){1}}
\put(317.5,69){\line(1,0){1}} \put(317.5,21.3){\line(1,0){1}}
\put(320.5,69){\line(1,0){1}} \put(320.5,21.3){\line(1,0){1}}
\put(323.5,69){\line(1,0){1}} \put(323.5,21.3){\line(1,0){1}}
\end{picture}
\end{center}
\vspace{-0.7cm} \caption{Transition rate diagram of $\{X(t):t\geq
0\}$.}
\end{figure}

We remark that this is a generalization of the model in
\cite{DicrescenzoMacciMartinucci}; in fact we recover that model
by setting
$$\left\{\begin{array}{ll}
\alpha_1=\lambda\eta+\mu(1-\eta)\\
\beta_1=\mu\eta+\lambda(1-\eta)\\
\alpha_2=\lambda\theta+\mu(1-\theta)\\
\beta_2=\mu\theta+\lambda(1-\theta)
\end{array}\right.$$
for $\lambda,\mu>0$ and $\eta,\theta\in[0,1]$; moreover the case
$(\theta,\eta)=(1,1)$ was studied in
\cite{DicrescenzoIulianoMartinucci}, whereas the case
$(\theta,\eta)=(0,1)$ identifies the model investigated in
\cite{ConollyParthasarathyDharmaraja} and
\cite{TarabiaTakagiElbaz}. 

In particular we extend the results in
\cite{DicrescenzoMacciMartinucci} by giving explicit expressions
of the probability generating function, mean and variance of
$X(t)$ (for each fixed $t>0$), and we study the asymptotic
behavior (as $t\to\infty$) in the fashion of large deviations.
Here we also give explicit expressions of the state probabilities.

Moreover we consider some random time-changes of the basic model
$\{X(t):t\geq 0\}$, with independent processes. This is motivated
by the great interest that the theory of random time-changes (and
subordination) is being receiving starting from \cite{Bochner}
(see also \cite{Schilling}). In particular this theory allows to
construct non-standard models which are useful for possible 
applications in different fields; indeed, in many circumstances, 
the process is more realistically assumed to evolve according to 
a random (so-called operational) time, instead of the usual 
deterministic one. A wide class of random time-changes concerns 
subordinators, namely nondecreasing L\'{e}vy processes (see, for
example, \cite{Sato}, \cite{KumarNaneVellaisamy},
\cite{MeerschaertNaneVellaisamy} and \cite{OrsingherBeghin},
\cite{DicrescenzoMartinucciZacks}); recent works with different
kind of random time-changes are \cite{DingGieseckeTomecek},
\cite{BeghinOrsingherJoSP2016} and \cite{DovidioOrsingherToaldo}.

The random time-changes of $\{X(t):t\geq 0\}$ studied in this
paper are related to fractional differential equations and stable
processes. More precisely we consider:
\begin{enumerate}
\item the inverse of the stable subordinator $\{T^\nu(t):t\geq
0\}$;
\item the (possibly tempered) stable subordinator
$\{\tilde{S}^{\nu,\mu}(t):t\geq 0\}$ for $\nu\in(0,1)$ and
$\mu\geq 0$ (we have the tempered case when $\mu>0$).
\end{enumerate}
In both cases, i.e. for both $\{X(T^\nu(t)):t\geq 0\}$ and 
$\{X(\tilde{S}^{\nu,\mu}(t)):t\geq 0\}$, we provide expressions 
for the state probabilities in terms of the generalized Fox-Wright 
function. We recall \cite{HoudreKawai}, \cite{Rosinski} and
\cite{SabzikarMeerschaertChen} among the references with the
tempered stable subordinator. Typically these two random
time-changes are associated to some generalized derivative in the
literature; namely the Caputo left fractional derivative (see, for
example, (2.4.14) and (2.4.15) in \cite{KilbasSrivastavaTrujillo})
in the first case, and the shifted fractional derivative (see (6)
in \cite{BeghinJCP}; see also (17) in \cite{BeghinJCP} for the
connections with the fractional Riemann-Liouville derivative) in
the second case.

We also try to extend the large deviation results for $\{X(t):t\geq 0\}$ 
to the cases with a random time-change considered in this paper. It is 
useful to remark that all the large deviation principles in this paper 
are proved by applications of the G\"{a}rtner Ellis Theorem; moreover 
these large deviation principles yield the convergence (at least in 
probability) to the values at which the large deviation rate functions 
uniquely vanish. Thus, motivated by potential applications, when dealing 
with large deviation principles with the same speed function, we compare 
the rate functions to establish if we have a faster or slower convergence
(if they are comparable). In conclusion the evaluation of the rate function
can be an important task, in particular when they are given in terms of a
variational formula (as happens with the application of the G\"{a}rtner 
Ellis Theorem).

The applications of the G\"{a}rtner Ellis Theorem are based on suitable 
limits of moment generating functions. So, in view of the applications of 
this theorem, we study the probability generating functions of the random
variables of the processes; in particular the formulas obtained for 
$\{X(T^\nu(t)):t\geq 0\}$ have some analogies with many results in the 
literature for other time-fractional processes (for instance the probability
generating functions are expressed in terms of the Mittag-Leffler function),
with both continuous and discrete state space (see, for example, 
\cite{MeerschaertNaneVellaisamy}, \cite{HahnKobayashiUmarov}, 
\cite{BeghinMacciJAP2014} and \cite{Iksanov-etal}). For 
$\{X(T^\nu(t)):t\geq 0\}$ we can consider large deviations only (the 
difficulties to obtain a moderate deviation result are briefly discussed);
moreover we compute (and plot) different large deviation rate functions for 
various choices of $\nu\in(0,1)$ and we conclude that, the smaller is $\nu$, 
the faster is the convergence of $\frac{X^\nu(t)}{t}$ to zero (as $t\to\infty$).
For $\{X(\tilde{S}^{\nu,\mu}(t)):t\geq 0\}$ we can obtain large and moderate deviations
for the tempered case $\mu>0$ only; in fact in this case we can apply the G\"{a}rtner 
Ellis Theorem because we have light-tailed distributed random variables (namely the 
moment generating functions of the involved random variables are finite in a 
neighborhood of the origin).

There are some references in the literature with applications of the G\"{a}rtner Ellis 
Theorem to time-changed processes. However there are very few cases where the random 
time-change is given by the inverse of the stable subordinator; see e.g. 
\cite{GajdaMagdziarz} and \cite{WangChang} where the time-changed processes are 
fractional Brownian motions. We are not aware of any other references where the 
time-changed process takes values on $\mathbb{Z}$.

We conclude with the outline of the paper. Section
\ref{sec:preliminaries} is devoted to some preliminaries on large
deviations. In Section \ref{sec:non-fractional} we present the
results for the basic model, i.e. the (non-fractional) process
$\{X(t):t\geq 0\}$. Finally we present some results for the
process $\{X(t):t\geq 0\}$ with random time-changes: the case with
the inverse of the stable subordinator is studied in Section
\ref{sec:time-fractional}, the case with the (possibly tempered)
stable subodinator is studied in Section
\ref{sec:time-change-TSS}. The final appendix (Section
\ref{sec:pmf-expressions}) is devoted to the state probabilities
expressions.

\section{Preliminaries on large deviations}\label{sec:preliminaries}
Some results in this paper concerns the theory of large
deviations; so, in this section, we recall some preliminaries (see
e.g. \cite{DemboZeitouni}, pages 4-5). A family of probability
measures $\{\pi_t:t>0\}$ on a topological space $\mathcal{Y}$
satisfies the large deviation principle (LDP for short) with rate
function $I$ and speed function $v_t$ if:
$\lim_{t\to+\infty}v_t=+\infty$, $I:\mathcal{Y}\to[0,+\infty]$ is
lower semicontinuous,
$$\liminf_{t\to+\infty}\frac{1}{v_t}\log\pi_t(O)\geq -\inf_{y\in O}I(y)$$
for all open sets $O$, and
$$\limsup_{t\to+\infty}\frac{1}{v_t}\log\pi_t(C)\leq -\inf_{y\in C}I(y)$$
for all closed sets $C$. A rate function is said to be good if all
its level sets $\{\{y\in\mathcal{Y}:I(y)\leq\eta\}:\eta\geq 0\}$
are compact.

We also present moderate deviation results. This terminology is
used when, for each family of positive numbers $\{a_t:t>0\}$ such
that $a_t\to 0$ and $ta_t\to\infty$, we have a family of laws of
centered random variables (which depend on $a_t$), which satisfies
the LDP with speed function $1/a_t$, and they are governed by the
same quadratic rate function which uniquely vanishes at zero (for
every choice of $\{a_t:t>0\}$). More precisely we have a rate
function $J(y)=\frac{y^2}{2\sigma^2}$, for some $\sigma^2>0$.
Typically moderate deviations fill the gap between a convergence
to zero of centered random variables, and a convergence in
distribution to a centered Normal distribution with variance
$\sigma^2$.

The main large deviation tool used in this paper is the
G\"{a}rtner Ellis Theorem (see e.g. Theorem 2.3.6 in
\cite{DemboZeitouni}).

\section{Results for the basic model (non-fractional case)}\label{sec:non-fractional}
In this section we present the results for the basic model. Some
of them will be used for the models with random time-changes in
the next sections. We start with some non-asymptotic results,
where $t$ is fixed, which concern probability generating
functions, means and variances. In the second part we present the
asymptotic results, namely large and (moderate) deviation results
as $t\to\infty$.

In particular the probability generating functions
$\{F_k(\cdot,t):k\in\mathbb{Z},t\geq 0\}$ are important in both
parts; they are defined by
$$F_k(z,t):=\mathbb{E}\left[z^{X(t)}|X(0)=k\right]=\sum_{n=-\infty}^\infty z^np_{k,n}(t)\ (\mbox{for}\ k\in\mathbb{Z}),$$
where $\{p_{k,n}(t):k,n\in\mathbb{Z},t\geq 0\}$ are the state
probabilities in \eqref{eq:pmf-notation}.

We also have to consider the function
$\Lambda:\mathbb{R}\to\mathbb{R}$ defined by
\begin{equation}\label{eq:def-Lambda}
\Lambda(\gamma):=\frac{h(e^\gamma)}{e^\gamma}-\frac{\alpha_1+\alpha_2+\beta_1+\beta_2}{2},
\end{equation}
where
\begin{equation}\label{eq:hz}
\left.\begin{array}{l}
h(z):=\frac{1}{2}\sqrt{\tilde{h}(z;\alpha_1,\alpha_2,\beta_1,\beta_2)},\ \mbox{where}\\
\tilde{h}(z;\alpha_1,\alpha_2,\beta_1,\beta_2):=(\alpha_1+\alpha_2-(\beta_1+\beta_2))^2z^2+4(\beta_1z^2+\beta_2)(\alpha_1z^2+\alpha_2).
\end{array}\right.
\end{equation}

\begin{remark}\label{rem:exchange-parameters}
The non-asymptotic results presented below depend on $k=X(0)$, and
we have different formulations when $k$ is odd or even. In
particular we can reduce from a case to another by exchanging
$(\alpha_1,\alpha_2)$ and $(\beta_1,\beta_2)$. On the contrary $k$
is negligible for the asymptotic results; in fact
$\tilde{h}(z;\alpha_1,\alpha_2,\beta_1,\beta_2)=\tilde{h}(z;\beta_1,\beta_2,\alpha_1,\alpha_2)$,
and we have an analogous property for the function $\Lambda$, for
its first derivative $\Lambda^\prime$ and its second derivative
$\Lambda^{\prime\prime}$.
\end{remark}

The function $\Lambda$ is the analogue of the function $\Lambda$
in equation (14) in \cite{DicrescenzoMacciMartinucci}, and plays a
crucial role in the proofs of the large (and moderate) deviation
results. However we refer to this function also for the
non-asymptotic results in order to have simpler expressions; in
particular we refer to the derivatives $\Lambda^\prime(0)$ and
$\Lambda^{\prime\prime}(0)$ and therefore we present the following
lemma.

\begin{lemma}\label{lem:2-derivatives-Lambda-origin}
Let $\Lambda$ be the function in \eqref{eq:def-Lambda}. Then we
have
$$\Lambda^\prime(0)=\frac{2(\alpha_1\beta_1-\alpha_2\beta_2)}{\alpha_1+\alpha_2+\beta_1+\beta_2}$$
and
$$\Lambda^{\prime\prime}(0)=\frac{4(\alpha_1\beta_1+\alpha_2\beta_2)}{\alpha_1+\alpha_2+\beta_1+\beta_2}
-\frac{8(\alpha_1\beta_1-\alpha_2\beta_2)^2}{(\alpha_1+\alpha_2+\beta_1+\beta_2)^3}.$$
Moreover $\Lambda^{\prime\prime}(0)>0$; in fact
$$\Lambda^{\prime\prime}(0)=\frac{4\{(\alpha_1\beta_1+\alpha_2\beta_2)[(\alpha_1+\alpha_2)^2+(\beta_1+\beta_2)^2
+2\alpha_1\beta_2+2\alpha_2\beta_1]+8\alpha_1\alpha_2\beta_1\beta_2\}}{(\alpha_1+\alpha_2+\beta_1+\beta_2)^3}.$$
\end{lemma}
\begin{proof}
We have
$$\Lambda^\prime(\gamma)=\frac{h^\prime(e^\gamma)e^{2\gamma}-e^\gamma h(e^\gamma)}{e^{2\gamma}}=h^\prime(e^\gamma)-e^{-\gamma}h(e^\gamma),$$
which yields $\Lambda^\prime(0)=h^\prime(1)-h(1)$, and
$$\Lambda^{\prime\prime}(\gamma)=h^{\prime\prime}(e^\gamma)e^\gamma-(-e^{-\gamma}h(e^\gamma)+h^\prime(e^\gamma))=
h^{\prime\prime}(e^\gamma)e^\gamma+e^{-\gamma}h(e^\gamma)-h^\prime(e^\gamma),$$
which yields
$\Lambda^{\prime\prime}(0)=h^{\prime\prime}(1)+h(1)-h^\prime(1)=h^{\prime\prime}(1)-\Lambda^\prime(0)$.

The desired equalities can be checked with some cumbersome
computations. In particular we can check that
$$h(1)=\frac{1}{2}\sqrt{(\alpha_1+\alpha_2-(\beta_1+\beta_2))^2+4(\beta_1+\beta_2)(\alpha_1+\alpha_2)}=\frac{\alpha_1+\alpha_2+\beta_1+\beta_2}{2}$$
and
\begin{align*}
h^\prime(1)=\left.\frac{1}{2}\frac{2(\alpha_1+\alpha_2-(\beta_1+\beta_2))^2z+4[2z\beta_1(\alpha_1z^2+\alpha_2)+2z\alpha_1(\beta_1z^2+\beta_2)]}
{2\sqrt{(\alpha_1+\alpha_2-(\beta_1+\beta_2))^2z^2+4(\beta_1z^2+\beta_2)(\alpha_1z^2+\alpha_2)}}\right|_{z=1}\\
=\frac{(\alpha_1+\alpha_2)^2+(\beta_1+\beta_2)^2+6\alpha_1\beta_1+2\alpha_1\beta_2+2\alpha_2\beta_1-2\alpha_2\beta_2}{2(\alpha_1+\alpha_2+\beta_1+\beta_2)};
\end{align*}
moreover, if we use the symbol $g$ in place of
$\tilde{h}(\cdot;\alpha_1,\alpha_2,\beta_1,\beta_2)$ in
\eqref{eq:hz} for simplicity, we can check that
$$g(1)=(\alpha_1+\alpha_2+\beta_1+\beta_2)^2,$$
$$g^\prime(1)=2[(\alpha_1+\alpha_2-(\beta_1+\beta_2))^2+8\alpha_1\beta_1+4\alpha_1\beta_2+4\alpha_2\beta_1],$$
$$g^{\prime\prime}(1)=2[(\alpha_1+\alpha_2-(\beta_1+\beta_2))^2+24\alpha_1\beta_1+4\alpha_1\beta_2+4\alpha_2\beta_1],$$
and
\begin{multline*}
h^{\prime\prime}(1)=\frac{2g(1)g^{\prime\prime}(1)-(g^\prime(1))^2}{8(g(1))^{3/2}}\\
=\frac{4(\alpha_1+\alpha_2+\beta_1+\beta_2)^2[(\alpha_1+\alpha_2-(\beta_1+\beta_2))^2+24\alpha_1\beta_1+4\alpha_1\beta_2+4\alpha_2\beta_1]}
{8(\alpha_1+\alpha_2+\beta_1+\beta_2)^3}\\
-\frac{4[(\alpha_1+\alpha_2-(\beta_1+\beta_2))^2+8\alpha_1\beta_1+4\alpha_1\beta_2+4\alpha_2\beta_1]^2}{8(\alpha_1+\alpha_2+\beta_1+\beta_2)^3}\\
=\frac{(\alpha_1+\alpha_2-(\beta_1+\beta_2))^2+24\alpha_1\beta_1+4\alpha_1\beta_2+4\alpha_2\beta_1}
{2(\alpha_1+\alpha_2+\beta_1+\beta_2)}\\
-\frac{[(\alpha_1+\alpha_2-(\beta_1+\beta_2))^2+8\alpha_1\beta_1+4\alpha_1\beta_2+4\alpha_2\beta_1]^2}{2(\alpha_1+\alpha_2+\beta_1+\beta_2)^3}.
\end{multline*}
Other details are omitted.
\end{proof}

\subsection{Non-asymptotic results}
In this section we present explicit formulas for probability
generating functions (see Proposition \ref{prop:pgf}), means and
variances (see Proposition \ref{prop:mean-variance}). In all these
propositions we can check what we said in Remark
\ref{rem:exchange-parameters} about the exchange of
$(\alpha_1,\alpha_2)$ and $(\beta_1,\beta_2)$.

In view of this we present some preliminaries. It is known that
the state probabilities solve the equations
$$\left\{\begin{array}{ll}
\dot{p}_{k,2n}(t)=\beta_1p_{k,2n-1}(t)-(\alpha_1+\alpha_2)p_{k,2n}(t)+\beta_2p_{k,2n+1}(t)\\
\dot{p}_{k,2n+1}(t)=\alpha_1p_{k,2n}(t)-(\beta_1+\beta_2)p_{k,2n+1}(t)+\alpha_2p_{k,2n+2}(t)\\
p_{k,n}(0)=1_{\{k=n\}}
\end{array}\right.\ (\mbox{for}\ k\in\mathbb{Z}).$$
So, if we consider the decomposition
\begin{equation}\label{eq:pgf-decomposition}
F_k=G_k+H_k,
\end{equation}
where $G_k$ and $H_k$ are the generating functions defined by
$$G_k(z,t):=\sum_{j=-\infty}^\infty z^{2j}p_{k,2j}(t)\ \mbox{and}\ H_k(z,t):=\sum_{j=-\infty}^\infty
z^{2j+1}p_{k,2j+1}(t)=\sum_{j=-\infty}^\infty
z^{2j-1}p_{k,2j-1}(t),$$ we have
\begin{equation}\label{eq:pgf-equations}
\left\{\begin{array}{ll}
\frac{\partial G_k(z,t)}{\partial t}=z\beta_1H_k(z,t)-(\alpha_1+\alpha_2)G_k(z,t)+\frac{\beta_2}{z}H_k(z,t)\\
\frac{\partial H_k(z,t)}{\partial t}=z\alpha_1G_k(z,t)-(\beta_1+\beta_2)H_k(z,t)+\frac{\alpha_2}{z}G_k(z,t)\\
G_k(z,0)=z^k\cdot 1_{\{k\ \mathrm{is\ even}\}},\ H_k(z,0)=z^k\cdot
1_{\{k\ \mathrm{is\ odd}\}}
\end{array}\right.\ (\mbox{for}\ k\in\mathbb{Z}).
\end{equation}
We remark that, if we consider the matrix
\begin{equation}\label{eq:def-matrix-A}
A:=\left(\begin{array}{cc}
-(\alpha_1+\alpha_2)&z\beta_1+\frac{\beta_2}{z}\\
z\alpha_1+\frac{\alpha_2}{z}&-(\beta_1+\beta_2)
\end{array}\right),
\end{equation}
the equations \eqref{eq:pgf-equations} can be rewritten as
$$\left\{\begin{array}{ll}
\frac{\partial}{\partial t}\left(\begin{array}{c}
G_k(z,t)\\
H_k(z,t)
\end{array}\right)=A\left(\begin{array}{c}
G_k(z,t)\\
H_k(z,t)
\end{array}\right)\\
G_k(z,0)=z^k\cdot 1_{\{k\ \mathrm{is\ even}\}},\ H_k(z,0)=z^k\cdot
1_{\{k\ \mathrm{is\ odd}\}}
\end{array}\right.\ (\mbox{for}\ k\in\mathbb{Z}).$$
Thus
\begin{equation}\label{eq:pgf-matrices-for-decomposition}
\left(\begin{array}{c}
G_{2k}(z,t)\\
H_{2k}(z,t)
\end{array}\right)=e^{At}\left(\begin{array}{c}
z^{2k}\\
0
\end{array}\right)\ \mbox{and}\ \left(\begin{array}{c}
G_{2k+1}(z,t)\\
H_{2k+1}(z,t)
\end{array}\right)=e^{At}\left(\begin{array}{c}
0\\
z^{2k+1}\\
\end{array}\right).
\end{equation}

We start with the probability generating functions.

\begin{proposition}\label{prop:pgf}
For $z>0$ we have
$$F_k(z,t)=z^ke^{-\frac{\alpha_1+\alpha_2+\beta_1+\beta_2}{2}t}\left(\cosh\left(\frac{th(z)}{z}\right)+
\frac{c_k(z)}{h(z)}\sinh\left(\frac{th(z)}{z}\right)\right),$$
where
\begin{equation}\label{eq:coefficients}
c_k(z):=\left\{\begin{array}{ll}
\frac{(\beta_1+\beta_2-(\alpha_1+\alpha_2))z}{2}+\alpha_1z^2+\alpha_2&\ \mbox{if}\ k\ \mbox{is even}\\
\frac{(\alpha_1+\alpha_2-(\beta_1+\beta_2))z}{2}+\beta_1z^2+\beta_2&\
\mbox{if}\ k\ \mbox{is odd}.
\end{array}\right.
\end{equation}
\end{proposition}
\begin{proof}
The main part of the proof consists of the computation of the
exponential matrix $e^{At}$, where $A$ is the matrix in
\eqref{eq:def-matrix-A}, and finally we easily conclude by taking
into account \eqref{eq:pgf-decomposition} and
\eqref{eq:pgf-matrices-for-decomposition}.

The eigenvalues of $A$ are
\begin{equation}\label{eq:eigenvalues}
\hat{h}_\pm(z):=-\frac{\alpha_1+\alpha_2+\beta_1+\beta_2}{2}\pm\frac{h(z)}{z}
\end{equation}
(where $h$ is defined by \eqref{eq:hz}), and it is known that we
can find a matrix $S$ such that
$$S\left(\begin{array}{cc}
-\frac{\alpha_1+\alpha_2+\beta_1+\beta_2}{2}-\frac{h(z)}{z}&0\\
0&-\frac{\alpha_1+\alpha_2+\beta_1+\beta_2}{2}+\frac{h(z)}{z}
\end{array}\right)S^{-1}=A;$$
in particular we can consider the matrix
$$S:=\left(\begin{array}{cc}
\frac{\beta_1+\beta_2-(\alpha_1+\alpha_2)}{2}-\frac{h(z)}{z}&\frac{\beta_1+\beta_2-(\alpha_1+\alpha_2)}{2}+\frac{h(z)}{z}\\
z\alpha_1+\frac{\alpha_2}{z}&z\alpha_1+\frac{\alpha_2}{z}
\end{array}\right)$$
and its inverse is
$$S^{-1}=-\frac{z}{2h(z)}\left(\begin{array}{cc}
1&\frac{-z[\beta_1+\beta_2-(\alpha_1+\alpha_2)]-2h(z)}{2(\alpha_1z^2+\alpha_2)}\\
-1&\frac{z[\beta_1+\beta_2-(\alpha_1+\alpha_2)]-2h(z)}{2(\alpha_1z^2+\alpha_2)}.
\end{array}\right).$$
Then the desired exponential matrix is
\begin{multline*}
e^{At}=S\left(\begin{array}{cc}
e^{(-\frac{\alpha_1+\alpha_2+\beta_1+\beta_2}{2}-\frac{h(z)}{z})t}&0\\
0&e^{(-\frac{\alpha_1+\alpha_2+\beta_1+\beta_2}{2}+\frac{h(z)}{z})t}
\end{array}\right)S^{-1}\\
=-\frac{z}{2h(z)}S\left(\begin{array}{cc}
e^{\hat{h}_-(z)t}&e^{\hat{h}_-(z)t}\cdot\frac{-z[\beta_1+\beta_2-(\alpha_1+\alpha_2)]-2h(z)}{2(\alpha_1z^2+\alpha_2)}\\
-e^{\hat{h}_+(z)t}&e^{\hat{h}_+(z)t}\cdot\frac{z[\beta_1+\beta_2-(\alpha_1+\alpha_2)]-2h(z)}{2(\alpha_1^2z+\alpha_2)}
\end{array}\right);
\end{multline*}
moreover, after some computations, we have
$$e^{At}=\left(\begin{array}{cc}
u_{11}(z,t)&u_{12}(z,t)\\
u_{21}(z,t)&u_{22}(z,t)
\end{array}\right),$$
where
\begin{multline*}
u_{11}(z,t):=-\frac{z}{2h(z)}\left(e^{\hat{h}_-(z)t}\left(\frac{\beta_1+\beta_2-(\alpha_1+\alpha_2)}{2}-\frac{h(z)}{z}\right)
-e^{\hat{h}_+(z)t}\left(\frac{\beta_1+\beta_2-(\alpha_1+\alpha_2)}{2}+\frac{h(z)}{z}\right)\right)\\
=\frac{(\beta_1+\beta_2-(\alpha_1+\alpha_2))z}{2h(z)}\cdot\frac{e^{\hat{h}_+(z)t}-e^{\hat{h}_-(z)t}}{2}+\frac{e^{\hat{h}_-(z)t}+e^{\hat{h}_+(z)t}}{2},
\end{multline*}
$$u_{21}(z,t):=-\frac{z}{2h(z)}\left(z\alpha_1+\frac{\alpha_2}{z}\right)(e^{\hat{h}_-(z)t}-e^{\hat{h}_+(z)t})
=\frac{\alpha_1z^2+\alpha_2}{h(z)}\cdot\frac{e^{\hat{h}_+(z)t}-e^{\hat{h}_-(z)t}}{2},$$
\begin{multline*}
u_{12}(z,t):=-\frac{z}{2h(z)}\left(\left(\frac{\beta_1+\beta_2-(\alpha_1+\alpha_2)}{2}-\frac{h(z)}{z}\right)
e^{\hat{h}_-(z)t}\cdot\frac{-z[\beta_1+\beta_2-(\alpha_1+\alpha_2)]-2h(z)}{2(\alpha_1z^2+\alpha_2)}\right.\\
\left.+\left(\frac{\beta_1+\beta_2-(\alpha_1+\alpha_2)}{2}+\frac{h(z)}{z}\right)
e^{\hat{h}_+(z)t}\cdot\frac{z[\beta_1+\beta_2-(\alpha_1+\alpha_2)]-2h(z)}{2(\alpha_1z^2+\alpha_2)}\right)\\
=-\frac{z}{2h(z)}\cdot\frac{z}{\alpha_1z^2+\alpha_2}\left(e^{\hat{h}_-(z)t}-e^{\hat{h}_+(z)t}\right)
\left(\frac{h^2(z)}{z^2}-\frac{(\beta_1+\beta_2-(\alpha_1+\alpha_2))^2}{4}\right)\\
=-\frac{1}{h(z)(\alpha_1z^2+\alpha_2)}\left(h^2(z)-\frac{(\beta_1+\beta_2-(\alpha_1+\alpha_2))^2z^2}{4}\right)
\frac{e^{\hat{h}_-(z)t}-e^{\hat{h}_+(z)t}}{2}\\
=\frac{\beta_1z^2+\beta_2}{h(z)}\cdot\frac{e^{\hat{h}_+(z)t}-e^{\hat{h}_-(z)t}}{2}
\end{multline*}
and
\begin{multline*}
u_{22}(z,t):=-\frac{z}{2h(z)}\left(z\alpha_1+\frac{\alpha_2}{z}\right)\\
\left(e^{\hat{h}_-(z)t}\cdot\frac{-z[\beta_1+\beta_2-(\alpha_1+\alpha_2)]-2h(z)}{2(\alpha_1z^2+\alpha_2)}
+e^{\hat{h}_+(z)t}\cdot\frac{z[\beta_1+\beta_2-(\alpha_1+\alpha_2)]-2h(z)}{2(\alpha_1z^2+\alpha_2)}\right)\\
=\frac{e^{\hat{h}_-(z)t}}{2}\left(\frac{(\beta_1+\beta_2-(\alpha_1+\alpha_2))z}{2h(z)}+1\right)
+\frac{e^{\hat{h}_+(z)t}}{2}\left(-\frac{(\beta_1+\beta_2-(\alpha_1+\alpha_2))z}{2h(z)}+1\right)\\
=\frac{(\beta_1+\beta_2-(\alpha_1+\alpha_2))z}{2h(z)}\cdot\frac{e^{\hat{h}_-(z)t}-e^{\hat{h}_+(z)t}}{2}+\frac{e^{\hat{h}_-(z)t}+e^{\hat{h}_+(z)t}}{2}.
\end{multline*}

We complete the proof noting that, by \eqref{eq:pgf-decomposition}
and \eqref{eq:pgf-matrices-for-decomposition}, we have
$$F_{2k}(z,t)=z^{2k}(u_{11}(z,t)+u_{21}(z,t))$$
and
$$F_{2k+1}(z,t)=z^{2k+1}(u_{12}(z,t)+u_{22}(z,t));$$
in fact these equalities yield
\begin{multline*}
F_{2k}(z,t)=z^{2k}\left(\frac{e^{\hat{h}_-(z)t}+e^{\hat{h}_+(z)t}}{2}
+\frac{1}{h(z)}\left(\frac{\beta_1+\beta_2-(\alpha_1+\alpha_2)}{2}z+\alpha_1z^2+\alpha_2\right)\frac{e^{\hat{h}_+(z)t}-e^{\hat{h}_-(z)t}}{2}\right)\\
=z^{2k}e^{-\frac{\alpha_1+\alpha_2+\beta_1+\beta_2}{2}t}\left(\cosh\left(\frac{th(z)}{z}\right)
+\frac{c_{2k}(z)}{h(z)}\sinh\left(\frac{th(z)}{z}\right)\right)
\end{multline*}
and
\begin{multline*}
F_{2k+1}(z,t)=z^{2k+1}\left(\frac{e^{\hat{h}_-(z)t}+e^{\hat{h}_+(z)t}}{2}
+\frac{1}{h(z)}\left(\frac{\alpha_1+\alpha_2-(\beta_1+\beta_2)}{2}z+\beta_1z^2+\beta_2\right)\frac{e^{\hat{h}_+(z)t}-e^{\hat{h}_-(z)t}}{2}\right)\\
=z^{2k+1}e^{-\frac{\alpha_1+\alpha_2+\beta_1+\beta_2}{2}t}\left(\cosh\left(\frac{th(z)}{z}\right)
+\frac{c_{2k+1}(z)}{h(z)}\sinh\left(\frac{th(z)}{z}\right)\right).
\end{multline*}
\end{proof}

In the next proposition we give mean and variance; in particular
we refer to $\Lambda^\prime(0)$ and $\Lambda^{\prime\prime}(0)$
given in Lemma \ref{lem:2-derivatives-Lambda-origin}.

\begin{proposition}\label{prop:mean-variance}
We have
$$\mathbb{E}[X(t)|X(0)=k]=k+\Lambda^\prime(0)t+\left.\left(\frac{c_k(z)}{h(z)}\right)^\prime\right|_{z=1}
\frac{1-e^{-(\alpha_1+\alpha_2+\beta_1+\beta_2)t}}{2},$$ where
$$\left.\left(\frac{c_k(z)}{h(z)}\right)^\prime\right|_{z=1}=\left\{\begin{array}{ll}
\frac{2(\alpha_1+\alpha_2)(\alpha_1-\alpha_2-\beta_1+\beta_2)}{(\alpha_1+\alpha_2+\beta_1+\beta_2)^2}&\ \mbox{if}\ k\ \mbox{is even}\\
\frac{2(\beta_1+\beta_2)(\beta_1-\beta_2-\alpha_1+\alpha_2)}{(\alpha_1+\alpha_2+\beta_1+\beta_2)^2}&\
\mbox{if}\ k\ \mbox{is odd}.
\end{array}\right.$$
Moreover, if $k$ is even, we have
\begin{multline*}
\mathrm{Var}[X(t)|X(0)=k]=\Lambda^{\prime\prime}(0)t-e^{-2(\alpha_1+\alpha_2+\beta_1+\beta_2)t}
\frac{(\alpha_1+\alpha_2)^2(\alpha_1-\alpha_2-\beta_1+\beta_2)^2}{(\alpha_1+\alpha_2+\beta_1+\beta_2)^4}\\
+\frac{e^{-(\alpha_1+\alpha_2+\beta_1+\beta_2)t}}{(\alpha_1+\alpha_2+\beta_1+\beta_2)^3}
\left\{8t(\alpha_1+\alpha_2)(\alpha_1-\alpha_2-\beta_1+\beta_2)(\alpha_1\beta_1-\alpha_2\beta_2)\right.\\
+(\alpha_1+\alpha_2)(\alpha_1-\alpha_2-\beta_1+\beta_2)(\alpha_1+\alpha_2-\beta_1-\beta_2)
-6(\alpha_2-\beta_1)(\alpha_1-\alpha_2-\beta_1+\beta_2)(\alpha_1+\alpha_2-\beta_1-\beta_2)\\
-2(7\alpha_2+\beta_1-2\beta_2)(\beta_1+\beta_2)(\alpha_1+\alpha_2-\beta_1-\beta_2)
-4(\alpha_2-\beta_2)^2(\alpha_1+\alpha_2-\beta_1-\beta_2)+8\alpha_2(\beta_1+\beta_2)(\alpha_1+\alpha_2)\\
\left.-8\alpha_2(\beta_1+\beta_2)(\alpha_1-\alpha_2-\beta_1+\beta_2)+8\beta_1(\beta_1+\beta_2)^2
-\frac{16(\alpha_2+\beta_1)^2(\beta_1+\beta_2)^2}{\alpha_1+\alpha_2+\beta_1+\beta_2}\right\}\\
+\frac{1}{(\alpha_1+\alpha_2+\beta_1+\beta_2)^3}\left\{(-7\alpha_1+3\alpha_2+10\beta_1-4\beta_2)(\beta_1+\beta_2)(\alpha_1-\alpha_2-\beta_1+\beta_2)\right.\\
+4(\alpha_2+\alpha_1)(\alpha_2+2\beta_2)(\alpha_1-\alpha_2-\beta_1+\beta_2)+4(\alpha_2-\beta_2)^2(\alpha_1-\alpha_2-\beta_1+\beta_2)\\
+4(\alpha_2-\beta_2)(\alpha_2+\beta_1)(\beta_1+\beta_2)-10(\alpha_2+\beta_1)(\beta_1+\beta_2)^2\\
\left.+8(\alpha_2+\beta_1)(\alpha_2-\beta_2)^2\right\}
+\frac{20(\alpha_2+\beta_1)^2(\beta_1+\beta_2)^2}{(\alpha_1+\alpha_2+\beta_1+\beta_2)^4}.
\end{multline*}
Finally, if $k$ is odd, $\mathrm{Var}[X(t)|X(0)=k]$ can be
obtained by exchanging $(\alpha_1,\alpha_2)$ and
$(\beta_1,\beta_2)$ in the last expression (we recall that, as
pointed out in Remark \ref{rem:exchange-parameters},
$\Lambda^{\prime\prime}(0)$ does not change).
\end{proposition}
\begin{proof}
The desired expressions of means and variance can be obtained with
suitable (well-known) formulas in terms of
$\left.\frac{dF_k(z,t)}{dz}\right|_{z=1}$ and
$\left.\frac{d^2F_k(z,t)}{dz^2}\right|_{z=1}$; these two values
can be computed by considering the explicit formulas of $F_k(z,t)$
in Proposition \ref{prop:pgf}. The computations are cumbersome and
we omit the details.
\end{proof}

\subsection{Asymptotic results}
Here we present Propositions \ref{prop:LD} and \ref{prop:MD},
which are the generalization of Propositions 3.1 and 3.2 in
\cite{DicrescenzoMacciMartinucci}. In both cases we apply the
G\"{artner} Ellis Theorem, and we use the probability generating
function in Proposition \ref{prop:pgf}. Actually the proof of
Proposition \ref{prop:MD} here is slightly different from the
proof of Proposition 3.2 in \cite{DicrescenzoMacciMartinucci}.

\begin{proposition}\label{prop:LD}
For all $k\in\mathbb{Z}$,
$\left\{P\left(\frac{X(t)}{t}\in\cdot\Big|X(0)=k\right):t>0\right\}$
satisfies the LDP with speed function $v_t=t$ and good rate
function $\Lambda^*(y):=\sup_{\gamma\in\mathbb{R}}\{\gamma
y-\Lambda(\gamma)\}$.
\end{proposition}
\begin{proof}
We can simply adapt the proof of Proposition 3.1 in
\cite{DicrescenzoMacciMartinucci}. The details are omitted.
\end{proof}

\begin{proposition}\label{prop:MD}
Let $\{a_t:t>0\}$ be such that $a_t\to 0$ and $ta_t\to+\infty$ (as
$t\to+\infty$). Then, for all $k\in\mathbb{Z}$,
$\left\{P\left(\sqrt{ta_t}\frac{X(t)-\mathbb{E}[X(t)|X(0)=k]}{t}\in\cdot\Big|X(0)=k\right):t>0\right\}$
satisfies the LDP with speed function $v_t=\frac{1}{a_t}$ and good
rate function $J(y):=\frac{y^2}{2\Lambda^{\prime\prime}(0)}$.
\end{proposition}
\begin{proof}
We apply the G\"{a}rtner Ellis Theorem. More precisely we show
that
\begin{equation}\label{eq:GE-limit-MD-non-fractional}
\lim_{t\to\infty}a_t\log\mathbb{E}\left[\exp\left(\frac{\gamma}{a_t}\sqrt{ta_t}\frac{X(t)-\mathbb{E}[X(t)|X(0)=k]}{t}\right)\Big|X(0)=k\right]=
\frac{\gamma^2}{2}\Lambda^{\prime\prime}(0)\ (\mbox{for all}\
\gamma\in\mathbb{R});
\end{equation}
in fact we can easily check that
$J(y)=\sup_{\gamma\in\mathbb{R}}\left\{\gamma
y-\frac{\gamma^2}{2}\Lambda^{\prime\prime}(0)\right\}$ (for all
$y\in\mathbb{R}$).

We remark that
\begin{multline*}
a_t\log\mathbb{E}\left[\exp\left(\frac{\gamma}{a_t}\sqrt{ta_t}\frac{X(t)-\mathbb{E}[X(t)|X(0)=k]}{t}\right)\Big|X(0)=k\right]\\
=a_t\left(\log\mathbb{E}\left[\exp\left(\frac{\gamma}{\sqrt{ta_t}}X(t)\right)\Big|X(0)=k\right]
-\frac{\gamma}{\sqrt{ta_t}}\mathbb{E}[X(t)|X(0)=k]\right).
\end{multline*}
As far as the right hand side is concerned, we take into account
Proposition \ref{prop:pgf} for the moment generating function and
Proposition \ref{prop:mean-variance} for the mean; then we get
\begin{multline*}
\lim_{t\to\infty}a_t\log\mathbb{E}\left[\exp\left(\frac{\gamma}{a_t}\sqrt{ta_t}\frac{X(t)-\mathbb{E}[X(t)|X(0)=k]}{t}\right)\Big|X(0)=k\right]\\
=\lim_{t\to\infty}a_t\left(k\frac{\gamma}{\sqrt{ta_t}}-\frac{\alpha_1+\alpha_2+\beta_1+\beta_2}{2}t+t\frac{h(e^{\gamma/\sqrt{ta_t}})}{e^{\gamma/\sqrt{ta_t}}}
-\frac{\gamma}{\sqrt{ta_t}}(k+\Lambda^\prime(0)t)\right)
\end{multline*}
and, by \eqref{eq:def-Lambda}, we obtain
$$\lim_{t\to\infty}a_t\log\mathbb{E}\left[\exp\left(\frac{\gamma}{a_t}\sqrt{ta_t}\frac{X(t)-\mathbb{E}[X(t)|X(0)=k]}{t}\right)\Big|X(0)=k\right]
=\lim_{t\to\infty}ta_t\left(\Lambda\left(\frac{\gamma}{\sqrt{ta_t}}\right)-\frac{\gamma}{\sqrt{ta_t}}\Lambda^\prime(0)\right).$$
Finally, if we consider the second order Taylor formula for the
function $\Lambda$, we have
$$\lim_{t\to\infty}ta_t\left(\Lambda\left(\frac{\gamma}{\sqrt{ta_t}}\right)-\frac{\gamma}{\sqrt{ta_t}}\Lambda^\prime(0)\right)
=\lim_{t\to\infty}ta_t\left(\frac{\gamma^2}{2ta_t}\Lambda^{\prime\prime}(0)+o\left(\frac{\gamma^2}{ta_t}\right)\right)$$
for a remainder $o\left(\frac{\gamma^2}{ta_t}\right)$ such that
$o\left(\frac{\gamma^2}{ta_t}\right)/\frac{\gamma^2}{ta_t}\to 0$,
and \eqref{eq:GE-limit-MD-non-fractional} is checked.
\end{proof}

\begin{remark}\label{rem:alternative-proof-MD}
The expressions of mean and variance in Proposition
\ref{prop:mean-variance} yield the following limits:
$$\lim_{t\to\infty}\frac{\mathbb{E}[X(t)|X(0)=k]}{t}=\Lambda^\prime(0);\ 
\lim_{t\to\infty}\frac{\mathrm{Var}[X(t)|X(0)=k]}{t}=\Lambda^{\prime\prime}(0).$$
These limits give a generalization of the analogue limits in
\cite{DicrescenzoMacciMartinucci}.
\end{remark}

\section{Results with the inverse of the stable subordinator}\label{sec:time-fractional}
In this section we consider the process $\{X^\nu(t):t\geq 0\}$,
for $\nu\in(0,1)$, i.e.
\begin{equation}\label{eq:time-change-representation-fractional}
X^\nu(t):=X^1(T^\nu(t)),
\end{equation}
where $\{T^\nu(t):t\geq 0\}$ is the inverse of the stable
subordinator, independent of a version of the non-fractional
process $\{X^1(t):t\geq 0\}$ studied above.

So we recall some preliminaries. We start with the definition of
the Mittag-Leffler function (see e.g. \cite{Podlubny}, page 17):
$$E_\nu(x):=\sum_{j\geq 0}\frac{x^j}{\Gamma(\nu j+1)}\ (\mbox{for all}\ x\in\mathbb{R}).$$
Then we have
\begin{equation}\label{eq:mgf-fractional}
\mathbb{E}[e^{\gamma T^\nu(t)}]=E_\nu(\gamma t^\nu).
\end{equation}
In some references this formula is stated assuming that
$\gamma\leq 0$ but this restriction is not needed because we can
refer to the analytic continuation of the Laplace transform with
complex argument. We also recall that formula (24) in
\cite{MainardiMuraPagnini} provides a version of
\eqref{eq:mgf-fractional} for $t=1$ (in that formula there is $-s$
in place $\gamma$, and $s\in\mathbb{C}$).

\subsection{Non-asymptotic results}
We start with Proposition \ref{prop:pgf-fractional}, which
provides an expression for the probability generating functions
$\{F_k^\nu(\cdot,t):k\in\mathbb{Z},t\geq 0\}$ defined by
$$F_k^\nu(z,t):=\mathbb{E}\left[z^{X^\nu(t)}|X^\nu(0)=k\right]
=\sum_{n=-\infty}^\infty z^n p_{k,n}^\nu(t)\ (\mbox{for}\
k\in\mathbb{Z}),$$ where $\{p_{k,n}^\nu(t):k,n\in\mathbb{Z},t\geq
0\}$ are the state probabilities defined by
\begin{equation}\label{eq:pmf-notation-fractional}
p_{k,n}^\nu(t):=P(X^\nu(t)=n|X^\nu(0)=k).
\end{equation}
Obviously Proposition \ref{prop:pgf-fractional} is the analogue of
Proposition \ref{prop:pgf} (and we can recover it by setting
$\nu=1$).

\begin{proposition}\label{prop:pgf-fractional}
For $z>0$ we have
$$F_k^\nu(z,t)=z^k\left(\frac{E_\nu(\hat{h}_-(z)t^\nu)+E_\nu(\hat{h}_+(z)t^\nu)}{2}
+\frac{c_k(z)}{h(z)}\cdot\frac{E_\nu(\hat{h}_+(z)t^\nu)-E_\nu(\hat{h}_-(z)t^\nu)}{2}\right),$$
where $c_k(z)$ is as in \eqref{eq:coefficients} and
$\hat{h}_\pm(z)$ are the eigenvalues in \eqref{eq:eigenvalues}.
\end{proposition}
\begin{proof}
We recall that $T^\nu(0)=0$. Then, if we refer the expression of
the probability generating functions
$\{F_k(\cdot,t):k\in\mathbb{Z},t\geq 0\}$ in Proposition
\ref{prop:pgf}, we have
\begin{multline*}
F_k^\nu(z,t)=\mathbb{E}\left[z^{X^1(T^\nu(t))}|X^1(0)=k\right]=\mathbb{E}\left[F_k(z,T^\nu(t))|X^1(0)=k\right]\\
=\mathbb{E}\left[z^ke^{-\frac{\alpha_1+\alpha_2+\beta_1+\beta_2}{2}T^\nu(t)}\left(\cosh\left(\frac{T^\nu(t)h(z)}{z}\right)+
\frac{c_k(z)}{h(z)}\sinh\left(\frac{T^\nu(t)h(z)}{z}\right)\right)|X^1(0)=k\right].
\end{multline*}
Then, by taking into account the moment generating function in
\eqref{eq:mgf-fractional}, after some manipulations we get
\begin{multline*}
\tilde{F}_k^{\nu,\mu}(z,t)=z^k\left(\left(1+\frac{c_k(z)}{h(z)}\right)\frac{\mathbb{E}[e^{\hat{h}_+(z)T^\nu(t)}]}{2}
+\left(1-\frac{c_k(z)}{h(z)}\right)\frac{\mathbb{E}[e^{\hat{h}_-(z)T^\nu(t)}]}{2}\right)\\
=z^k\left(\left(1+\frac{c_k(z)}{h(z)}\right)\frac{E_\nu(\hat{h}_+(z)t^\nu)}{2}
+\left(1-\frac{c_k(z)}{h(z)}\right)\frac{E_\nu(\hat{h}_-(z)t^\nu)}{2}\right).
\end{multline*}
So we can immediately check that this coincides with the
expression in the statement of the proposition.
\end{proof}

\subsection{Asymptotic results}
Here we present Proposition \ref{prop:LD-fractional}, which is the
analogue of Proposition \ref{prop:LD}. Unfortunately we cannot 
present a moderate deviation result, namely we cannot present the 
analogue of Proposition \ref{prop:MD}; see the discussion in Remark
\ref{rem:MD-fractional-missing}.

We conclude this section with Remark \ref{rem:comparison}, where
we compare the convergence of processes for different values of
$\nu\in (0,1)$. In fact, if we consider the framework of
Proposition \ref{prop:LD-fractional} below, the rate function
$\Lambda_\nu^*(y)$ uniquely vanishes at $y=0$,
and therefore $\frac{X^\nu(t)}{t}$ converges to $0$ as
$t\to\infty$ (we recall that, for $\nu=1$, $\frac{X^\nu(t)}{t}$ 
converges to $\Lambda^\prime(0)$ as $t\to\infty$); moreover, the
more $\Lambda_\nu^*(y)$ is larger around $y=0$, the more the 
convergence of $\frac{X^\nu(t)}{t}$ is faster. In particular in 
Remark \ref{rem:comparison} we take $0<\nu_1<\nu_2<1$, and we 
get strict inequalities between
$\Lambda_{\nu_1}^*(y)$ and $\Lambda_{\nu_2}^*(y)$ in a
sufficiently small neighborhood of the origin $y=0$ (except the
origin itself because we have
$\Lambda_{\nu_1}^*(0)=\Lambda_{\nu_2}^*(0)=0$).

\begin{proposition}\label{prop:LD-fractional}
We set
$$\Lambda_\nu(\gamma):=\left\{\begin{array}{ll}
(\Lambda(\gamma))^{1/\nu}&\ \mbox{if}\ \Lambda(\gamma)\geq 0\\
0&\ \mbox{if}\ \Lambda(\gamma)<0,
\end{array}\right.$$	
where $\Lambda$ is the function in \eqref{eq:def-Lambda}. Then, for all $k\in\mathbb{Z}$,
$\left\{P\left(\frac{X^\nu(t)}{t}\in\cdot\Big|X^\nu(0)=k\right):t>0\right\}$
satisfies the LDP with speed function $v_t=t$ and good rate
function $\Lambda_\nu^*(y):=\sup_{\gamma\in\mathbb{R}}\{\gamma
y-\Lambda_\nu(\gamma)\}$.
\end{proposition}
\begin{proof}
We want to apply the G\"{a}rtner Ellis Theorem and, for all
$\gamma\in\mathbb{R}$, we have to take the limit of
$\frac{1}{t}\log F_k^\nu(e^\gamma,t)$ (as $\to\infty$). Obviously
we consider the expression of the function $F_k^\nu(z,t)$ in
Proposition \ref{prop:pgf-fractional}.

Firstly, if $\nu\in(0,1)$, we have
\begin{equation}\label{eq:GE-limit-fractional}
\lim_{t\to\infty}\frac{1}{t}\log F_k^\nu(e^\gamma,t)=\Lambda_\nu(\gamma)\ (\mbox{for all}\ \gamma\in\mathbb{R});
\end{equation}
this can be checked noting that $\hat{h}_-(z)<0$,
$\hat{h}_+(e^\gamma)=\Lambda(\gamma)$ (for all
$\gamma\in\mathbb{R}$), by taking into account the limit
$$\lim_{t\to\infty}\frac{1}{t}\log E_\nu(ct^\nu)=\left\{\begin{array}{ll}
0&\ \mbox{if}\ c\leq 0\\
c^{1/\nu}&\ \mbox{if}\ c>0
\end{array}\right.$$
(this limit can be seen as a consequence of an expansion of
Mittag-Leffler function; see (1.8.27) in
\cite{KilbasSrivastavaTrujillo} with $\alpha=\nu$ and $\beta=1$),
and by considering a suitable application of Lemma 1.2.15 in
\cite{DemboZeitouni}.

Moreover the function $\Lambda_\nu$ in the limit \eqref{eq:GE-limit-fractional}
is nonnegative and attains its minimum, equal to zero, at the points of the set
$\{\gamma\in\mathbb{R}:\Lambda(\gamma)\leq 0\}$; we recall that this set can be
reduced to the single point $\gamma=0$ if and only if $\Lambda^\prime(0)=0$. 
Thus we can apply the G\"{a}rtner Ellis Theorem (because the function in the 
limit is finite everywhere and differentiable), and the desired LDP holds.
\end{proof}

\begin{remark}\label{rem:MD-fractional-missing}
We have some difficulties to get the extension of Proposition
\ref{prop:MD} for the time-fractional case. In fact, if a 
moderate deviation holds, we expect that it is governed by the
rate function $J_\nu(y):=\frac{y^2}{2\Lambda^{\prime\prime}(0)}$,
where $\Lambda^{\prime\prime}(0)$ is the second derivative at the
origin $\gamma=0$ of $\Lambda_\nu$, and assuming that such value exists
and it is finite. On the contrary $\Lambda^{\prime\prime}(0)$ exists
only if $\nu\in(0,1/2]$, and it is equal to zero. So, in such a case,
we should have
$$J_\nu(y):=\left\{\begin{array}{ll}
0&\ \mbox{if}\ y=0\\
\infty&\ \mbox{if}\ y\neq 0,
\end{array}\right.$$
and this rate function is not interesting; in fact it is the
largest rate function that we have for a sequence that converges
to zero (for instance this rate function comes up when we have
constant random variables converging to zero).
\end{remark}

\begin{remark}\label{rem:comparison}
We take $0<\nu_1<\nu_2<1$. We recall that:
\begin{itemize}
\item for $\nu\in(0,1)$ and $y\in\mathbb{R}$, the equation $\Lambda_\nu^\prime(\gamma)=y$
admits a solution; for the case $y=0$ we have
$$\{\gamma\in\mathbb{R}:\Lambda_\nu^\prime(\gamma)=0\}=\{\gamma\in\mathbb{R}:\Lambda(\gamma)\leq 0\},$$
and therefore we have a unique solution $\gamma=0$ if and only if $\Lambda^\prime(0)=0$; on the 
contrary, if $y\neq 0$, we have a unique solution $\gamma_{y,\nu}\in\mathbb{R}$, say;
\item there exists $\delta>0$ such that, if $\inf\{|\gamma-\tilde{\gamma}|:\Lambda(\tilde{\gamma})\leq 0\}<\delta$, 
then $\Lambda(\gamma)\in(0,1)$, and therefore
$0<\Lambda_{\nu_1}(\gamma)<\Lambda_{\nu_2}(\gamma)$.
\end{itemize}
Thus, by combining these two statements, there exists
$\delta^\prime>0$ such that, for $0<|y|<\delta^\prime$, we have
$$0<\Lambda_{\nu_2}^*(y)=\gamma_{y,\nu_2}y-\Lambda_{\nu_2}(\gamma_{y,\nu_2})<\gamma_{y,\nu_2}y-\Lambda_{\nu_1}(\gamma_{y,\nu_2})
\leq\sup_{\gamma\in\mathbb{R}}\{\gamma
y-\Lambda_{\nu_1}(\gamma)\}=\Lambda_{\nu_1}^*(y)$$ (see Figure
\ref{fig2} where $\Lambda^\prime(0)=0$ and we consider some specific values of $\nu$). In conclusion we
can say that
\begin{equation}\label{eq:comparison-statement}
\mbox{$\frac{X^{\nu_1}(t)}{t}$ converges to zero faster than
$\frac{X^{\nu_2}(t)}{t}$ (as $t\to\infty$).}
\end{equation}

We also remark that the statement \eqref{eq:comparison-statement}
is not surprising if we take into account the time-change
representation \eqref{eq:time-change-representation-fractional}.
In fact, if we denote the stable subordinator by $\{S^\nu(t):t\geq
0\}$, we have that
\begin{equation}\label{eq:fgm-stable-subordinator}
\mathbb{E}[e^{\gamma S^\nu(t)}]=\left\{\begin{array}{ll}
e^{-|\gamma|^\nu t}&\ \mbox{if}\ \gamma\leq 0\\
\infty&\ \mbox{otherwise};
\end{array}\right.
\end{equation}
thus, as $\nu\in(0,1)$ decreases, the increasing trend of
$\{S^\nu(t):t\geq 0\}$ increases, and therefore the increasing
trend of the inverse of the stable subordinator $\{T^\nu(t):t\geq
0\}$ decreases. Then, for $0<\nu_1<\nu_2<1$, the increasing
trend of the random time-change $\{T^{\nu_1}(t):t\geq 0\}$ for
$X(\cdot)$ is slower than the increasing trend of
$\{T^{\nu_2}(t):t\geq 0\}$; so $\frac{X^1(T^{\nu_1}(t))}{t}$
converges to zero faster than $\frac{X^1(T^{\nu_2}(t))}{t}$ (as
$t\to\infty$), and this statement meets
\eqref{eq:comparison-statement}.
\end{remark}

\begin{figure}[ht]
\begin{center}
\includegraphics[angle=0,width=0.60\textwidth]{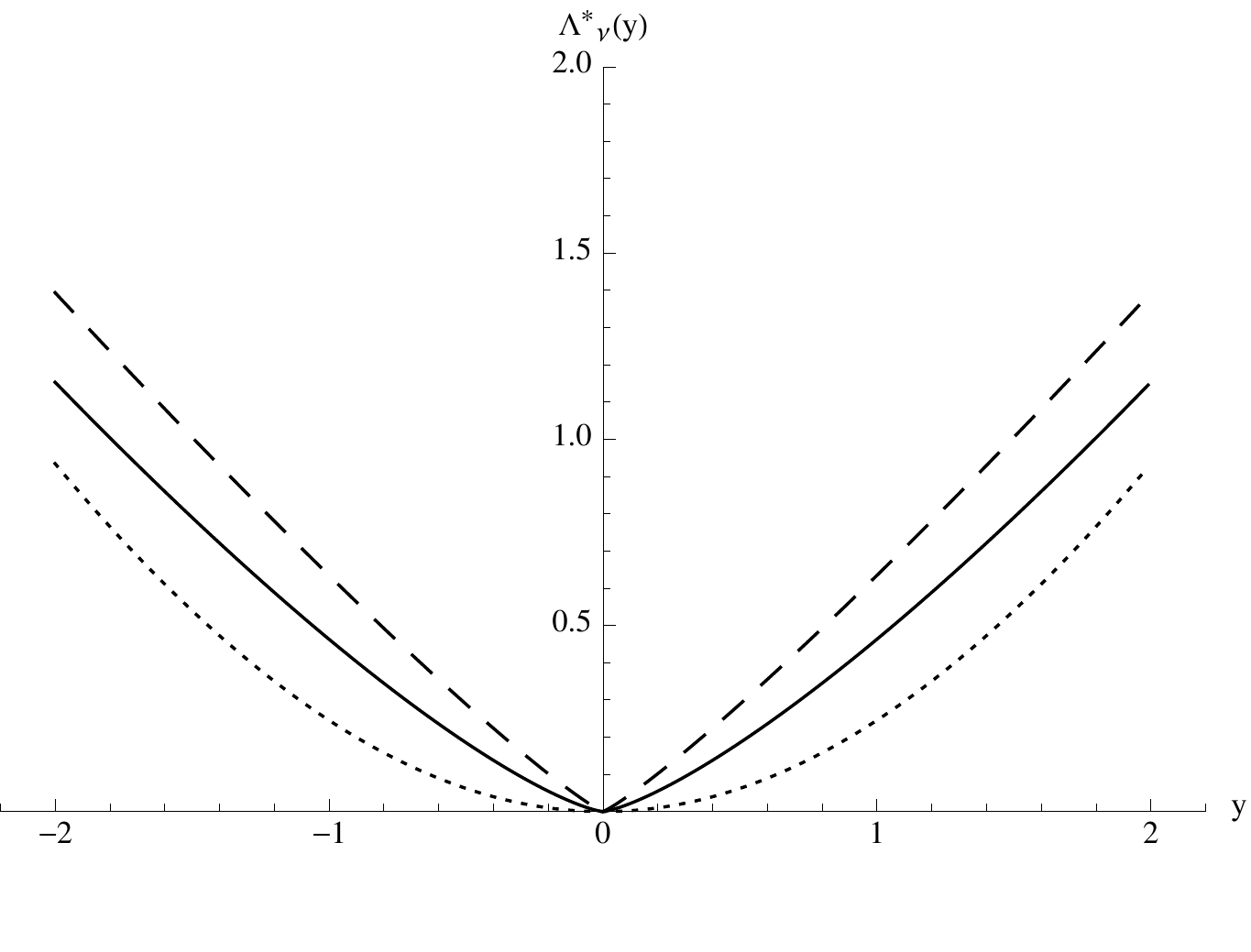}
\caption{The rate function $\Lambda_\nu^*$ around $y=0$ for
$\Lambda^\prime(0)=0$ (only in this case $\Lambda_\nu^*$ is 
differentiable everywhere; on the contrary, for $y=0$, left
and right hand derivatives of $\Lambda_\nu^*(y)$ do not coincide)
and some values	for $\nu$: $\nu=1/4$ (dashed line), $\nu=1/2$ 
(continuous line) and $\nu=1$ (dotted line).}\label{fig2}
\end{center}
\end{figure}

\section{Results with the (possibly tempered) stable subordinator}\label{sec:time-change-TSS}
In this section we consider the process
$\{\tilde{X}^{\nu,\mu}(t):t\geq 0\}$, for $\nu\in(0,1)$ and
$\mu\geq 0$, i.e.
$$\tilde{X}^{\nu,\mu}(t):=X^1(\tilde{S}^{\nu,\mu}(t)),$$
where $\{\tilde{S}^{\nu,\mu}(t):t\geq 0\}$ is a (possibly
tempered) stable subordinator, independent of a version of the
non-fractional process $\{X^1(t):t\geq 0\}$ studied above.

So we recall some preliminaries on $\{\tilde{S}^{\nu,\mu}(t):t\geq
0\}$. Firstly, for $t>0$, we have
$$P(\tilde{S}^{\nu,\mu}(t)\in dx)=\underbrace{e^{-\mu x+\mu^\nu t}f_{S^\nu(t)}(x)}_{=:f_{\tilde{S}^{\nu,\mu}(t)}(x)}dx,$$
where
$$P(S^\nu(t)\in dx)=f_{S^\nu(t)}(x)dx$$
and $\{S^\nu(t):t\geq 0\}$ is the stable subordinator; note that
$\{\tilde{S}^{\nu,\mu}(t):t\geq 0\}$ with $\mu=0$ coincides with
$\{S^\nu(t):t\geq 0\}$. Moreover we have
\begin{equation}\label{eq:mgf-TSS}
\mathbb{E}[e^{\gamma\tilde{S}^{\nu,\mu}(t)}]=e^{\mu^\nu
t}\mathbb{E}[e^{(\gamma-\mu )S^\nu(t)}]=\left\{
\begin{array}{ll}
e^{-t((\mu-\gamma)^\nu-\mu^\nu)}&\ \mbox{if}\ \gamma\leq\mu\\
\infty&\ \mbox{otherwise},
\end{array}\right.
\end{equation}
where we take into account \eqref{eq:fgm-stable-subordinator}.
Moreover, for $\mu>0$, if we consider the function
$\Psi_{\nu,\mu}$ defined by
\begin{equation}\label{eq:def-Psi}
\Psi_{\nu,\mu}(\gamma):=\left\{\begin{array}{ll}
\mu^\nu-(\mu-\gamma)^\nu&\ \mbox{if}\ \gamma\leq\mu\\
\infty&\ \mbox{otherwise},
\end{array}\right.
\end{equation}
for all $t>0$ we have
\begin{equation}\label{eq:Psi-means}
\frac{\mathbb{E}[\tilde{S}^{\nu,\mu}(t)]}{t}=\frac{\nu\mu^{\nu-1}t}{t}=\nu\mu^{\nu-1}=\Psi_{\nu,\mu}^\prime(0)
\end{equation}
and
\begin{equation}\label{eq:Psi-variances}
\frac{\mathrm{Var}[\tilde{S}^{\nu,\mu}(t)]}{t}=\frac{-\nu(\nu-1)\mu^{\nu-2}t}{t}=-\nu(\nu-1)\mu^{\nu-2}=\Psi_{\nu,\mu}^{\prime\prime}(0)
\end{equation}
(actually, if $\mu=0$, the above formulas \eqref{eq:Psi-means} and
\eqref{eq:Psi-variances} hold as left derivatives equal to
infinity).

\subsection{Non-asymptotic results}
We start with Proposition \ref{prop:pgf-TSS}, which provides an
expression for the probability generating functions
$\{\tilde{F}_k^{\nu,\mu}(\cdot,t):k\in\mathbb{Z},t\geq 0\}$
defined by
$$\tilde{F}_k^{\nu,\mu}(z,t):=\mathbb{E}\left[z^{\tilde{X}^{\nu,\mu}(t)}|\tilde{X}^{\nu,\mu}(0)=k\right]
=\sum_{n=-\infty}^\infty z^n\tilde{p}_{k,n}^{\nu,\mu}(t)\
(\mbox{for}\ k\in\mathbb{Z}),$$ where
$\{\tilde{p}_{k,n}^{\nu,\mu}(t):k,n\in\mathbb{Z},t\geq 0\}$ are
the state probabilities defined by
\begin{equation}\label{eq:pmf-notation-TSS}
\tilde{p}_{k,n}^{\nu,\mu}(t):=P(\tilde{X}^{\nu,\mu}(t)=n|\tilde{X}^{\nu,\mu}(0)=k).
\end{equation}
Obviously Proposition \ref{prop:pgf-TSS} is the analogue of
Propositions \ref{prop:pgf} and \ref{prop:pgf-fractional}. The
condition $\hat{h}_+(z)\leq\mu$ will be discussed after the proof.

\begin{proposition}\label{prop:pgf-TSS}
For $z>0$ we have
$$\tilde{F}_k^{\nu,\mu}(z,t)=\left\{\begin{array}{ll}
z^k\left(\frac{e^{-t((\mu-\hat{h}_+(z))^\nu-\mu^\nu)}+e^{-t((\mu-\hat{h}_-(z))^\nu-\mu^\nu)}}{2}\right.\\
\ \ \
+\left.\frac{c_k(z)}{h(z)}\cdot\frac{e^{-t((\mu-\hat{h}_+(z))^\nu-\mu^\nu)}-e^{-t((\mu-\hat{h}_-(z))^\nu-\mu^\nu)}}{2}\right)&\
\mbox{if}\ \hat{h}_+(z)\leq\mu\\
\infty&\ \mbox{otherwise},
\end{array}\right.$$
where $c_k(z)$ is as in \eqref{eq:coefficients} and
$\hat{h}_\pm(z)$ are the eigenvalues in \eqref{eq:eigenvalues}.
\end{proposition}
\begin{proof}
We recall that $\tilde{S}^{\nu,\mu}(0)=0$. Then, if we refer the
expression of the probability generating functions
$\{F_k(\cdot,t):k\in\mathbb{Z},t\geq 0\}$ in Proposition
\ref{prop:pgf}, we have
\begin{multline*}
\tilde{F}_k^{\nu,\mu}(z,t)=\mathbb{E}\left[z^{X^1(\tilde{S}^{\nu,\mu}(t))}|X^1(0)=k\right]=\mathbb{E}\left[F_k(z,\tilde{S}^{\nu,\mu}(t))|X^1(0)=k\right]\\
=\mathbb{E}\left[z^ke^{-\frac{\alpha_1+\alpha_2+\beta_1+\beta_2}{2}\tilde{S}^{\nu,\mu}(t)}\left(\cosh\left(\frac{\tilde{S}^{\nu,\mu}(t)h(z)}{z}\right)+
\frac{c_k(z)}{h(z)}\sinh\left(\frac{\tilde{S}^{\nu,\mu}(t)h(z)}{z}\right)\right)|X^1(0)=k\right].
\end{multline*}
Then, by taking into account the moment generating function in
\eqref{eq:mgf-TSS}, after some manipulations we get (we recall
that $\hat{h}_-(z)<0$)
\begin{multline*}
\tilde{F}_k^{\nu,\mu}(z,t)=z^k\left(\left(1+\frac{c_k(z)}{h(z)}\right)\frac{\mathbb{E}[e^{\hat{h}_+(z)\tilde{S}^{\nu,\mu}(t)}]}{2}
+\left(1-\frac{c_k(z)}{h(z)}\right)\frac{\mathbb{E}[e^{\hat{h}_-(z)\tilde{S}^{\nu,\mu}(t)}]}{2}\right)\\
=z^k\left(\left(1+\frac{c_k(z)}{h(z)}\right)\frac{e^{-t((\mu-\hat{h}_+(z))^\nu-\mu^\nu)}}{2}
+\left(1-\frac{c_k(z)}{h(z)}\right)\frac{e^{-t((\mu-\hat{h}_-(z))^\nu-\mu^\nu)}}{2}\right)
\end{multline*}
if $\hat{h}_+(z)\leq\mu$ (and infinity otherwise). So we can
easily check that this coincides with the expression in the
statement of the proposition.
\end{proof}

We conclude this section with a brief discussion on the condition
$\hat{h}_+(z)\leq\mu$ for $\mu\geq 0$. For $z>0$ we have
$$\frac{\sqrt{(\alpha_1+\alpha_2-(\beta_1+\beta_2))^2z^2+4(\beta_1z^2+\beta_2)(\alpha_1z^2+\alpha_2)}}{z}-\frac{\alpha_1+\alpha_2+\beta_1+\beta_2}{2}\leq\mu$$
by \eqref{eq:eigenvalues} and \eqref{eq:hz}. Then, after some easy
computations, it is easy to check that this is equivalent to
$$\alpha_1\beta_1z^4-(\mu^2+\mu(\alpha_1+\alpha_2+\beta_1+\beta_2)+\alpha_1\beta_1+\alpha_2\beta_2)z^2+\alpha_2\beta_2\leq 0;$$
in conclusion we have $\hat{h}_+(z)\leq\mu$ if and only if
$\sqrt{m_-(\mu)}\leq z\leq\sqrt{m_+(\mu)}$, where
\begin{multline*}
m_\pm(\mu):=\frac{1}{2\alpha_1\beta_1}\left\{\mu^2+\mu(\alpha_1+\alpha_2+\beta_1+\beta_2)+\alpha_1\beta_1+\alpha_2\beta_2\right.\\
\left.\pm\sqrt{(\mu^2+\mu(\alpha_1+\alpha_2+\beta_1+\beta_2)+\alpha_1\beta_1+\alpha_2\beta_2)^2-4\alpha_1\alpha_2\beta_1\beta_2}\right\}.
\end{multline*}
In particular, for case $\mu=0$, we have $\hat{h}_+(z)\leq 0$ if
and only if
$\sqrt{\min\left\{1,\frac{\alpha_2\beta_2}{\alpha_1\beta_1}\right\}}\leq
z\leq\sqrt{\max\left\{1,\frac{\alpha_2\beta_2}{\alpha_1\beta_1}\right\}}$
because
$$m_\pm(0)=\frac{\alpha_1\beta_1+\alpha_2\beta_2\pm |\alpha_1\beta_1-\alpha_2\beta_2|}{2\alpha_1\beta_1};$$
so we have $m_-(0)=1$ and/or $m_+(0)=1$, and they are both equal
to 1 if and only if $\alpha_1\beta_1=\alpha_2\beta_2$ or,
equivalently, $\Lambda^\prime(0)=0$ by Lemma
\ref{lem:2-derivatives-Lambda-origin}.

\subsection{Asymptotic results}
Here we present Proposition \ref{prop:LD-TSS}, which is the
analogue of Propositions \ref{prop:LD} and
\ref{prop:LD-fractional}. In this case we have no restriction on
the value of $\Lambda^\prime(0)$. Finally we present Proposition
\ref{prop:MD-TSS}, which is the analogue of Proposition
\ref{prop:MD}. In the proofs of Propositions \ref{prop:LD-TSS} and
\ref{prop:MD-TSS} we apply the G\"{a}rtner Ellis Theorem, and the
condition $\mu>0$ is required.

\begin{proposition}\label{prop:LD-TSS}
Assume that $\mu>0$, and set
$$\tilde{\Lambda}_{\nu,\mu}(\gamma):=\left\{\begin{array}{ll}
\mu^\nu-(\mu-\Lambda(\gamma))^\nu&\ \mbox{if}\ \Lambda(\gamma)\leq\mu\\
\infty&\ \mbox{otherwise},
\end{array}\right.$$
where $\Lambda$ is the function in \eqref{eq:def-Lambda}. Then,
for all $k\in\mathbb{Z}$,
$\left\{P\left(\frac{\tilde{X}^{\nu,\mu}(t)}{t}\in\cdot\Big|\tilde{X}^{\nu,\mu}(0)=k\right):t>0\right\}$
satisfies the LDP with speed function $v_t=t$ and good rate
function
$\tilde{\Lambda}_{\nu,\mu}^*(y):=\sup_{\gamma\in\mathbb{R}}\{\gamma
y-\tilde{\Lambda}_{\nu,\mu}(\gamma)\}$.
\end{proposition}
\begin{proof}
We want to apply the G\"{a}rtner Ellis Theorem and, for all
$\gamma\in\mathbb{R}$, we have to take the limit of
$\frac{1}{t}\log\tilde{F}_k^{\nu,\mu}(e^\gamma,t)$ (as
$\to\infty$). Obviously we consider the expression of the function
$\tilde{F}_k^{\nu,\mu}(z,t)$ in Proposition \ref{prop:pgf-TSS}.

Firstly we have
\begin{equation}\label{eq:GE-limit-TSS}
\lim_{t\to\infty}\frac{1}{t}\log\tilde{F}_k^{\nu,\mu}(e^\gamma,t)=\tilde{\Lambda}_{\nu,\mu}(\gamma)\
(\mbox{for all}\ \gamma\in\mathbb{R});
\end{equation}
this can be checked noting that $\hat{h}_-(z)<0$,
$\hat{h}_+(e^\gamma)=\Lambda(\gamma)$ (for all
$\gamma\in\mathbb{R}$), and by considering a suitable application
of Lemma 1.2.15 in \cite{DemboZeitouni}.

The function $\tilde{\Lambda}_{\nu,\mu}$ in the limit
\eqref{eq:GE-limit-TSS} is essentially smooth (see e.g. Definition
2.3.5 in \cite{DemboZeitouni}); in fact it is finite in a
neighborhood of the origin, differentiable in the interior of the
set
$\mathcal{D}:=\{\gamma\in\mathbb{R}:\tilde{\Lambda}_{\nu,\mu}(\gamma)<\infty\}$,
and steep (namely
$\tilde{\Lambda}_{\nu,\mu}^\prime(\gamma_n)\to\infty$ for every
sequence $\{\gamma_n:n\geq 1\}$ in the interior of $\mathcal{D}$
which converges to a boundary point of the interior of
$\mathcal{D}$) because, if $\gamma_0$ is such that
$\Lambda(\gamma_0)=\mu$, we have
$$\tilde{\Lambda}_{\nu,\mu}^\prime(\gamma)=\nu(\mu-\Lambda(\gamma))^{\nu-1}\Lambda^\prime(\gamma)\to\infty\ (\mbox{as}\ \gamma\to\gamma_0).$$
Then we can apply the G\"{a}rtner Ellis Theorem (in fact the
function $\tilde{\Lambda}_{\nu,\mu}$ is also lower
semi-continuous), and the desired LDP holds.
\end{proof}

In view of the next result on moderate deviations we compute
$\tilde{\Lambda}_{\nu,\mu}^{\prime\prime}(0)$. We remark that, if
we consider the function $\Psi_{\nu,\mu}$ in \eqref{eq:def-Psi},
we have
$\tilde{\Lambda}_{\nu,\mu}(\gamma)=\Psi_{\nu,\mu}(\Lambda(\gamma))$
(for all $\gamma\in\mathbb{R}$). Thus we have
$$\tilde{\Lambda}_{\nu,\mu}^\prime(\gamma)=\Psi_{\nu,\mu}^\prime(\Lambda(\gamma))\Lambda^\prime(\gamma),\
\tilde{\Lambda}_{\nu,\mu}^{\prime\prime}(\gamma)=\Psi_{\nu,\mu}^\prime(\Lambda(\gamma))\Lambda^{\prime\prime}(\gamma)
+\Psi_{\nu,\mu}^{\prime\prime}(\Lambda(\gamma))(\Lambda^\prime(\gamma))^2$$
and therefore (for the second equality see \eqref{eq:Psi-means}
and \eqref{eq:Psi-variances})
$$\tilde{\Lambda}_{\nu,\mu}^{\prime\prime}(0)=\Psi_{\nu,\mu}^\prime(0)\Lambda^{\prime\prime}(0)
+\Psi_{\nu,\mu}^{\prime\prime}(0)(\Lambda^\prime(0))^2=\nu\mu^{\nu-1}\Lambda^{\prime\prime}(0)-\nu(\nu-1)\mu^{\nu-2}(\Lambda^\prime(0))^2.$$
We remark that $\tilde{\Lambda}_{\nu,\mu}^{\prime\prime}(0)>0$ because $\Lambda^{\prime\prime}(0)>0$
(see Lemma \ref{lem:2-derivatives-Lambda-origin}) and $\mu>0$.

\begin{proposition}\label{prop:MD-TSS}
Assume that $\mu>0$. Let $\{a_t:t>0\}$ be such that $a_t\to 0$ and
$ta_t\to+\infty$ (as $t\to+\infty$). Then, for all
$k\in\mathbb{Z}$,
$\left\{P\left(\sqrt{ta_t}\frac{\tilde{X}^{\nu,\mu}(t)-\mathbb{E}[\tilde{X}^{\nu,\mu}(t)|\tilde{X}^{\nu,\mu}(0)=k]}{t}\in\cdot
\Big|\tilde{X}^{\nu,\mu}(0)=k\right):t>0\right\}$ satisfies the
LDP with speed function $v_t=\frac{1}{a_t}$ and good rate function
$J_{\nu,\mu}(y):=\frac{y^2}{2\tilde{\Lambda}_{\nu,\mu}^{\prime\prime}(0)}$.
\end{proposition}
\begin{proof}
We apply the G\"{a}rtner Ellis Theorem. More precisely we show
that
\begin{equation}\label{eq:GE-limit-MD-TSS}
\lim_{t\to\infty}a_t\log\mathbb{E}\left[\exp\left(\frac{\gamma}{a_t}\sqrt{ta_t}\frac{\tilde{X}^{\nu,\mu}(t)
-\mathbb{E}[\tilde{X}^{\nu,\mu}(t)|\tilde{X}^{\nu,\mu}(0)=k]}{t}\right)\Big|\tilde{X}^{\nu,\mu}(0)=k\right]=
\frac{\gamma^2}{2}\tilde{\Lambda}_{\nu,\mu}^{\prime\prime}(0)\ (\mbox{for all}\
\gamma\in\mathbb{R});
\end{equation}
in fact we can easily check that
$J_{\nu,\mu}(y)=\sup_{\gamma\in\mathbb{R}}\left\{\gamma
y-\frac{\gamma^2}{2}\tilde{\Lambda}_{\nu,\mu}^{\prime\prime}(0)\right\}$ (for all
$y\in\mathbb{R}$).

We remark that
\begin{multline*}
a_t\log\mathbb{E}\left[\exp\left(\frac{\gamma}{a_t}\sqrt{ta_t}\frac{\tilde{X}^{\nu,\mu}(t)
	-\mathbb{E}[\tilde{X}^{\nu,\mu}(t)|\tilde{X}^{\nu,\mu}(0)=k]}{t}\right)\Big|\tilde{X}^{\nu,\mu}(0)=k\right]\\
=a_t\left(\log\mathbb{E}\left[\exp\left(\frac{\gamma}{\sqrt{ta_t}}\tilde{X}^{\nu,\mu}(t)\right)\Big|\tilde{X}^{\nu,\mu}(0)=k\right]
-\frac{\gamma}{\sqrt{ta_t}}\mathbb{E}[\tilde{X}^{\nu,\mu}(t)|\tilde{X}^{\nu,\mu}(0)=k]\right)\\
=a_t\left(\log\tilde{F}_k^{\nu,\mu}(e^{\gamma/\sqrt{ta_t}},t)
-\frac{\gamma}{\sqrt{ta_t}}\mathbb{E}[\tilde{X}^{\nu,\mu}(t)|\tilde{X}^{\nu,\mu}(0)=k]\right),
\end{multline*}
where $\tilde{F}_k^{\nu,\mu}(z,t)$ is the probability generating function
in Proposition \ref{prop:pgf-TSS}. Moreover, by Proposition 
\ref{prop:mean-variance} (together with a conditioning with respect to 
$\{\tilde{S}^{\nu,\mu}(t):t\geq 0\}$ and some properties of this process)
we have
$$\mathbb{E}[\tilde{X}^{\nu,\mu}(t)|\tilde{X}^{\nu,\mu}(0)=k]=
k+\Lambda^\prime(0)\mathbb{E}[\tilde{S}^{\nu,\mu}(t)]+\mathbb{E}[b(\tilde{S}^{\nu,\mu}(t))],$$
where
$b(r)=\left.\left(\frac{c_k(z)}{h(z)}\right)^\prime\right|_{z=1}\frac{1-e^{-(\alpha_1+\alpha_2+\beta_1+\beta_2)r}}{2}$
is a bounded function of $r\geq 0$; thus, by \eqref{eq:Psi-means}, we have
$$\mathbb{E}[\tilde{X}^{\nu,\mu}(t)|\tilde{X}^{\nu,\mu}(0)=k]=
k+\Lambda^\prime(0)\Psi_{\nu,\mu}^\prime(0)t+\mathbb{E}[b(\tilde{S}^{\nu,\mu}(t))].$$
Then, since $\hat{h}_-(e^\gamma)<0$ and $\hat{h}_+(e^\gamma)=\Lambda(\gamma)$ for all 
$\gamma\in\mathbb{R}$, we get
\begin{multline*}
\lim_{t\to\infty}a_t\left(\log\tilde{F}_k^{\nu,\mu}(e^{\gamma/\sqrt{ta_t}},t)
-\frac{\gamma}{\sqrt{ta_t}}\mathbb{E}[\tilde{X}^{\nu,\mu}(t)|\tilde{X}^{\nu,\mu}(0)=k]\right)\\
\lim_{t\to\infty}a_t\left(k\frac{\gamma}{\sqrt{ta_t}}+t\tilde{\Lambda}_{\nu,\mu}\left(\frac{\gamma}{\sqrt{ta_t}}\right)
-\frac{\gamma}{\sqrt{ta_t}}\left(k+\Lambda^\prime(0)\Psi_{\nu,\mu}^\prime(0)t+\mathbb{E}[b(\tilde{S}^{\nu,\mu}(t))]\right)\right)\\
=\lim_{t\to\infty}ta_t\left(\tilde{\Lambda}_{\nu,\mu}\left(\frac{\gamma}{\sqrt{ta_t}}\right)
-\frac{\gamma}{\sqrt{ta_t}}\Lambda^\prime(0)\Psi_{\nu,\mu}^\prime(0)\right);
\end{multline*}
in fact the term with $\mathbb{E}[b(\tilde{S}^{\nu,\mu}(t))]$ is negligible because it is
the function $b(\cdot)$ is bounded. Finally, if we consider the second order Taylor formula
for the function $\tilde{\Lambda}_{\nu,\mu}$, we have
\begin{multline*}
\tilde{\Lambda}_{\nu,\mu}\left(\frac{\gamma}{\sqrt{ta_t}}\right)
-\frac{\gamma}{\sqrt{ta_t}}\Lambda^\prime(0)\Psi_{\nu,\mu}^\prime(0)\\
=\frac{\gamma}{\sqrt{ta_t}}\tilde{\Lambda}_{\nu,\mu}^\prime(0)+\frac{\gamma^2}{2ta_t}\tilde{\Lambda}_{\nu,\mu}^{\prime\prime}(0)
+o\left(\frac{\gamma^2}{ta_t}\right)-\frac{\gamma}{\sqrt{ta_t}}\Lambda^\prime(0)\Psi_{\nu,\mu}^\prime(0)
=\frac{\gamma^2}{2ta_t}\tilde{\Lambda}_{\nu,\mu}^{\prime\prime}(0)+o\left(\frac{\gamma^2}{ta_t}\right)
\end{multline*}
for a remainder $o\left(\frac{\gamma^2}{ta_t}\right)$ such that
$o\left(\frac{\gamma^2}{ta_t}\right)/\frac{\gamma^2}{ta_t}\to 0$,
and \eqref{eq:GE-limit-MD-TSS} can be easily checked.
\end{proof}	

\appendix
\section{State probabilities}\label{sec:pmf-expressions}
In this section we present some formulas for the state
probabilities \eqref{eq:pmf-notation},
\eqref{eq:pmf-notation-fractional} and
\eqref{eq:pmf-notation-TSS}. These formulas can be obtained by
extracting suitable coefficients of the probability generating
functions above; see Propositions \ref{prop:pgf},
\ref{prop:pgf-fractional} and \ref{prop:pgf-TSS}, respectively.
Here, as usual, binomial coefficients with negative arguments are
equal to zero. For each family of state probabilities we
distinguish two cases, and we introduce a suitable auxiliary
function: if $\alpha_1+\alpha_2\neq\beta_1+\beta_2$,
\begin{eqnarray*}
&& \hspace*{-1.5cm}
\vartheta^n_{r,s}(\alpha_1,\alpha_2,\beta_1,\beta_2):=\left(\frac{4
\alpha_2 \beta_2}{
(\alpha_1+\alpha_2-\beta_1-\beta_2)^2}\right)^{s-r}
\,\sum_{h=0}^{n-s+r} {n\choose h+s-r} \left(\frac{4 \alpha_1
\beta_2}{ (\alpha_1+\alpha_2-\beta_1-\beta_2)^2} \right)^{h}
\nonumber
\\
&& \hspace*{-0.8cm} \times \sum_{l=0}^{h} {h+s-r \choose l} {h+s-r
\choose h-l}  \left( \frac{\alpha_2 \beta_1}{\alpha_1 \beta_2}
\right)^{l};
\end{eqnarray*}
if $\alpha_1+\alpha_2=\beta_1+\beta_2$,
$$\eta^n_{r,s}(\alpha_1,\alpha_2,\beta_1,\beta_2):=\left(\frac{\alpha_2}{\alpha_1}\right)^{s-r}
(\alpha_1 \beta_2)^n \sum_{l=0}^{n-s+r} {n \choose l} {n \choose
s-r+l}  \left( \frac{\alpha_2 \beta_1}{\alpha_1 \beta_2}
\right)^{l}.$$

\begin{proposition}\label{prop:pmf-expressions}
Let $\{p_{k,n}(t):k,n\in\mathbb{Z},t\geq 0\}$ be as in
\eqref{eq:pmf-notation}.\\
(i) Assume that $\alpha_1+\alpha_2\neq\beta_1+\beta_2$. Then, for
all $s,r\in\mathbb{Z}$, we have the following four cases:
\begin{eqnarray*}
&& \hspace*{-1.2cm} p_{2r,2s}(t) = {\rm
e}^{-\frac{(\alpha_1+\alpha_2+\beta_1+\beta_2)}{2}\, t}
\sum_{n=|s-r|}^{+\infty} \left( \frac{\alpha_1 \beta_1}{\alpha_2
\beta_2}  \right)^{s-r}
\\
&& \hspace*{-1.2cm} \times \left[\frac{t^{2n}}{(2n)!}
\left(\frac{\beta_1+\beta_2-\alpha_1-\alpha_2}{2} \right)^{2
n}+\frac{t^{2n+1}}{(2n+1)!}
\left(\frac{\beta_1+\beta_2-\alpha_1-\alpha_2}{2} \right)^{2
n+1}\right] \cdot
\vartheta^n_{r,s}(\alpha_1,\alpha_2,\beta_1,\beta_2);
\end{eqnarray*}
\begin{eqnarray*}
&& \hspace*{-1.2cm} p_{2r,2s+1}(t) = {\rm
e}^{-\frac{(\alpha_1+\alpha_2+\beta_1+\beta_2)}{2}\, t} \left\{
\alpha_1 \sum_{n=|s-r|}^{+\infty}  \frac{t^{2n+1}}{(2n+1)!}
\left(\frac{\beta_1+\beta_2-\alpha_1-\alpha_2}{2} \right)^{2 n}
\left( \frac{\alpha_1 \beta_1}{\alpha_2 \beta_2}
\right)^{s-r}\!\!\cdot
\vartheta^n_{r,s}(\alpha_1,\alpha_2,\beta_1,\beta_2) \right.
\\
&& \hspace*{-1.2cm} \left. +\alpha_2 \sum_{n=|s-r+1|}^{+\infty}
\frac{t^{2n+1}}{(2n+1)!}
\left(\frac{\beta_1+\beta_2-\alpha_1-\alpha_2}{2} \right)^{2 n}
\left( \frac{\alpha_1 \beta_1}{\alpha_2 \beta_2}
\right)^{s-r+1}\!\!\cdot
\vartheta^n_{r,s+1}(\alpha_1,\alpha_2,\beta_1,\beta_2)\right\};
\end{eqnarray*}
\begin{eqnarray*}
&& \hspace*{-1.2cm} p_{2r+1,2s}(t) = {\rm
e}^{-\frac{(\alpha_1+\alpha_2+\beta_1+\beta_2)}{2}\, t} \left\{
\beta_2 \sum_{n=|s-r|}^{+\infty}  \frac{t^{2n+1}}{(2n+1)!}
\left(\frac{\beta_1+\beta_2-\alpha_1-\alpha_2}{2} \right)^{2 n}
\left( \frac{\alpha_1 \beta_1}{\alpha_2 \beta_2}
\right)^{s-r}\!\!\cdot
\vartheta^n_{r,s}(\beta_1,\beta_2,\alpha_1,\alpha_2) \right.
\\
&& \hspace*{-1.2cm} \left. +\beta_1 \sum_{n=|s-r-1|}^{+\infty}
\frac{t^{2n+1}}{(2n+1)!}
\left(\frac{\beta_1+\beta_2-\alpha_1-\alpha_2}{2} \right)^{2 n}
\left( \frac{\alpha_1 \beta_1}{\alpha_2 \beta_2}
\right)^{s-r-1}\!\!\cdot
\vartheta^n_{r,s-1}(\beta_1,\beta_2,\alpha_1,\alpha_2)\right\};
\end{eqnarray*}
\begin{eqnarray*}
&& \hspace*{-2cm} p_{2r+1,2s+1}(t) = {\rm
e}^{-\frac{(\alpha_1+\alpha_2+\beta_1+\beta_2)}{2}\, t}
\sum_{n=|s-r|}^{+\infty} \left( \frac{\alpha_1 \beta_1}{\alpha_2
\beta_2}  \right)^{s-r}
\\
&& \hspace*{-2cm} \times \left[\frac{t^{2n}}{(2n)!}
\left(\frac{\beta_1+\beta_2-\alpha_1-\alpha_2}{2} \right)^{2
n}+\frac{t^{2n+1}}{(2n+1)!}
\left(\frac{\alpha_1+\alpha_2-\beta_1-\beta_2}{2} \right)^{2
n+1}\right] \cdot
\vartheta^n_{r,s}(\beta_1,\beta_2,\alpha_1,\alpha_2).
\end{eqnarray*}
(ii) Assume that $\alpha_1+\alpha_2=\beta_1+\beta_2$ Then, for all
$s,r\in\mathbb{Z}$, we have the following four cases:
$$p_{2r,2s}(t) = {\rm e}^{-{(\alpha_1+\alpha_2)}\, t}
\sum_{n=|s-r|}^{+\infty} \left( \frac{\alpha_1 \beta_1}{\alpha_2
\beta_2}  \right)^{s-r} \frac{t^{2n}}{(2n)!} \cdot
\eta^n_{r,s}(\alpha_1,\alpha_2,\beta_1,\beta_2),$$
\begin{eqnarray*}
&& \hspace*{-1.2cm} p_{2r,2s+1}(t) =  {\rm
e}^{-{(\alpha_1+\alpha_2)}\, t} \left\{ \alpha_1
\sum_{n=|s-r|}^{+\infty}  \frac{t^{2n+1}}{(2n+1)!}  \left(
\frac{\alpha_1 \beta_1}{\alpha_2 \beta_2}  \right)^{s-r}\!\!\cdot
\eta^n_{r,s}(\alpha_1,\alpha_2,\beta_1,\beta_2) \right.
\\
&& \hspace*{-1.2cm} \left. +\alpha_2 \sum_{n=|s-r+1|}^{+\infty}
\frac{t^{2n+1}}{(2n+1)!}  \left( \frac{\alpha_1 \beta_1}{\alpha_2
\beta_2}  \right)^{s-r+1}\!\!\cdot
\eta^n_{r,s+1}(\alpha_1,\alpha_2,\beta_1,\beta_2)\right\},
\end{eqnarray*}
\begin{eqnarray*}
&& \hspace*{-1.2cm} p_{2r+1,2s}(t) =  {\rm
e}^{-{(\alpha_1+\alpha_2)}\, t} \left\{ \beta_2
\sum_{n=|s-r|}^{+\infty}  \frac{t^{2n+1}}{(2n+1)!} \left(
\frac{\alpha_1 \beta_1}{\alpha_2 \beta_2}  \right)^{s-r}\!\!\cdot
\eta^n_{r,s}(\beta_1,\beta_2,\alpha_1,\alpha_2) \right.
\\
&& \hspace*{-1.2cm} \left. +\beta_1 \sum_{n=|s-r-1|}^{+\infty}
\frac{t^{2n+1}}{(2n+1)!}  \left( \frac{\alpha_1 \beta_1}{\alpha_2
\beta_2}  \right)^{s-r-1}\!\!\cdot
\eta^n_{r,s-1}(\beta_1,\beta_2,\alpha_1,\alpha_2)\right\},
\end{eqnarray*}
$$p_{2r+1,2s+1}(t) =  {\rm e}^{-{(\alpha_1+\alpha_2)}\, t}
\sum_{n=|s-r|}^{+\infty} \frac{t^{2n}}{(2n)!} \left(
\frac{\alpha_1 \beta_1}{\alpha_2 \beta_2}  \right)^{s-r} \!\!\cdot
\eta^n_{r,s}(\beta_1,\beta_2,\alpha_1,\alpha_2).$$
\end{proposition}

\begin{remark}\label{rem:link-with-DIM}
If $\alpha_1=\alpha_2=\lambda$ and $\beta_1=\beta_2=\mu$, then
Proposition \ref{prop:pmf-expressions} coincides with Proposition
1 in \cite{DicrescenzoIulianoMartinucci} and corrects a misprint
contained in formula (18).
\end{remark}

In view of the next propositions we recall the definition of the
generalized Fox-Wright function (see e.g. (1.11.14) in
\cite{KilbasSrivastavaTrujillo}). We have
\begin{equation}\label{ppsiq}
{}_{p}\psi_{q} \left[\begin{array}{cc}
(a_l,\alpha_l)_{1,p} &  \\
                                    & ; z\\
(b_l,\beta_l)_{1,q}  &
\end{array}\right]
=\sum_{n=0}^{+\infty}\frac{z^n}{n!} \frac{\prod_{j=1}^p
\Gamma(a_j+\alpha_j n)} {\prod_{l=1}^q \Gamma(b_l+\beta_l n)},
\end{equation}
where $z,a_j,b_l\in {\mathbb{C}}$ and $\alpha_j,\beta_l\in
{\mathbb{R}}$.

\begin{proposition}\label{prop:pmf-expressions-fractional}
Let $\{p_{k,n}^\nu(t):k,n\in\mathbb{Z},t\geq 0\}$ be as in
\eqref{eq:pmf-notation-fractional}.\\
(i) Assume that $\alpha_1+\alpha_2\neq\beta_1+\beta_2$. Then, for
all $s,r\in\mathbb{Z}$, we have the following four cases:
\begin{eqnarray*}
&& \hspace*{-1.2cm} p^{\nu}_{2r,2s}(t) =\sum_{k=|s-r|}^{+\infty}
\left( \frac{\alpha_1 \beta_1}{\alpha_2 \beta_2}  \right)^{s-r}
\vartheta^k_{r,s}(\alpha_1,\alpha_2,\beta_1,\beta_2)
 \left(\frac{\beta_1+\beta_2-\alpha_1-\alpha_2}{2} \right)^{2 k}
\\
&& \hspace*{-1.2cm} \times \left\{ \frac{t^{2 k \nu}
(\alpha_1+\alpha_2-\beta_1-\beta_2)}{
(\alpha_1+\alpha_2+\beta_1+\beta_2) (2k+1)!} \,{}_{2}\psi_{2}
\left[\begin{array}{ccc}
(2k+1,2)\!\! & \!\! (1,1) & \\
&                                                                   & ; \frac{[t^\nu (\alpha_1+\alpha_2+\beta_1+\beta_2)]^2}{4} \\
(0,2)\!\! & \!\!(2 k \nu+1,2 \nu) &
\end{array}\right]
\right.
\\
&& \hspace*{-1.2cm} \left. +\frac{t^{2 k \nu}}{(2k)!}
\,{}_{2}\psi_{2} \left[\begin{array}{ccc}
(2k+1,2)\!\! &\!\! (1,1) & \\
&                                                                   & ; \frac{[t^\nu (\alpha_1+\alpha_2+\beta_1+\beta_2)]^2}{4} \\
(1,2)\!\! &\!\! (2 k \nu+1,2 \nu) &
\end{array}\right]
\right.
\\
&& \hspace*{-1.2cm} \left. -\frac{t^{(2 k+1) \nu}
(\alpha_1+\alpha_2-\beta_1-\beta_2)}{2 (2k+1)!} \,{}_{2}\psi_{2}
\left[\begin{array}{ccc}
(2k+2,2)\!\! & \!\! (1,1) & \\
&                                                                   & ; \frac{[t^\nu (\alpha_1+\alpha_2+\beta_1+\beta_2)]^2}{4} \\
(1,2)\!\! & \!\!((2 k+1) \nu+1,2 \nu) &
\end{array}\right]
\right.
\\
&& \hspace*{-1.2cm} \left. -\frac{t^{(2 k+1) \nu}
(\alpha_1+\alpha_2+\beta_1+\beta_2)}{2 (2k)!} \,{}_{2}\psi_{2}
\left[\begin{array}{ccc}
(2k+2,2)\!\! & \!\! (1,1) & \\
&                                                                   & ;\frac{[t^\nu (\alpha_1+\alpha_2+\beta_1+\beta_2)]^2}{4} \\
(2,2)\!\! & \!\!((2 k+1) \nu+1,2 \nu) &
\end{array}\right]
\right\};
\end{eqnarray*}
\begin{eqnarray*}
&& \hspace*{-1.2cm} p^{\nu}_{2r,2s+1}(t) = \alpha_1
\sum_{k=|s-r|}^{+\infty} \left( \frac{\alpha_1 \beta_1}{\alpha_2
\beta_2}  \right)^{s-r}
\vartheta^k_{r,s}(\alpha_1,\alpha_2,\beta_1,\beta_2)
 \left(\frac{\beta_1+\beta_2-\alpha_1-\alpha_2}{2} \right)^{2 k}
\\
&& \hspace*{-1.2cm} \times \left\{ \frac{t^{(2 k+1) \nu}} {
(2k+1)!} \,{}_{2}\psi_{2} \left[\begin{array}{ccc}
(2k+2,2)\!\! & \!\! (1,1) & \\
&                                                                   & ; \frac{[t^\nu (\alpha_1+\alpha_2+\beta_1+\beta_2)]^2}{4} \\
(1,2)\!\! & \!\!((2 k+1) \nu+1,2 \nu) &
\end{array}\right]
\right.
\\
&& \hspace*{-1.2cm} \left. -\frac{2 t^{2 k
\nu}}{(\alpha_1+\alpha_2+\beta_1+\beta_2)(2k+1)!} \,{}_{2}\psi_{2}
\left[\begin{array}{ccc}
(2k+1,2)\!\! &\!\! (1,1) & \\
&                                                                   & ; \frac{[t^\nu (\alpha_1+\alpha_2+\beta_1+\beta_2)]^2}{4} \\
(0,2)\!\! &\!\! (2 k \nu+1,2 \nu) &
\end{array}\right]
\right\}
\\
&& \hspace*{-1.2cm} +\alpha_2 \sum_{k=|s-r+1|}^{+\infty} \left(
\frac{\alpha_1 \beta_1}{\alpha_2 \beta_2}  \right)^{s-r+1}
\vartheta^k_{r,s+1}(\alpha_1,\alpha_2,\beta_1,\beta_2)
 \left(\frac{\beta_1+\beta_2-\alpha_1-\alpha_2}{2} \right)^{2 k}
\\
&& \hspace*{-1.2cm} \times \left\{ \frac{t^{(2 k+1) \nu}}{(2k+1)!}
\,{}_{2}\psi_{2} \left[\begin{array}{ccc}
(2k+2,2)\!\! & \!\! (1,1) & \\
&                                                                   & ; \frac{[t^\nu (\alpha_1+\alpha_2+\beta_1+\beta_2)]^2}{4} \\
(1,2)\!\! & \!\!((2 k+1) \nu+1,2 \nu) &
\end{array}\right]
\right.
\\
&& \hspace*{-1.2cm} \left. -\frac{2 t^{2 k \nu}
}{(\alpha_1+\alpha_2+\beta_1+\beta_2) (2k+1)!} \,{}_{2}\psi_{2}
\left[\begin{array}{ccc}
(2k+1,2)\!\! & \!\! (1,1) & \\
&                                                                   & ; \frac{[t^\nu (\alpha_1+\alpha_2+\beta_1+\beta_2)]^2}{4} \\
(0,2)\!\! & \!\!(2 k \nu+1,2 \nu) &
\end{array}\right]
\right\};
\end{eqnarray*}
\begin{eqnarray*}
&& \hspace*{-1.2cm} p^{\nu}_{2r+1,2s}(t) = \beta_2
\sum_{k=|s-r|}^{+\infty} \left( \frac{\alpha_1 \beta_1}{\alpha_2
\beta_2}  \right)^{s-r}
\vartheta^k_{r,s}(\beta_1,\beta_2,\alpha_1,\alpha_2)
 \left(\frac{\alpha_1+\alpha_2-\beta_1-\beta_2}{2} \right)^{2 k}
\\
&& \hspace*{-1.2cm} \times \left\{ \frac{t^{(2 k+1) \nu}} {
(2k+1)!} \,{}_{2}\psi_{2} \left[\begin{array}{ccc}
(2k+2,2)\!\! & \!\! (1,1) & \\
&                                                                   & ; \frac{[t^\nu (\alpha_1+\alpha_2+\beta_1+\beta_2)]^2}{4} \\
(1,2)\!\! & \!\!((2 k+1) \nu+1,2 \nu) &
\end{array}\right]
\right.
\\
&& \hspace*{-1.2cm} \left. -\frac{2 t^{2 k
\nu}}{(\alpha_1+\alpha_2+\beta_1+\beta_2)(2k+1)!} \,{}_{2}\psi_{2}
\left[\begin{array}{ccc}
(2k+1,2)\!\! &\!\! (1,1) & \\
&                                                                   & ; \frac{[t^\nu (\alpha_1+\alpha_2+\beta_1+\beta_2)]^2}{4} \\
(0,2)\!\! &\!\! (2 k \nu+1,2 \nu) &
\end{array}\right]
\right\}
\\
&& \hspace*{-1.2cm} +\beta_1 \sum_{k=|s-r-1|}^{+\infty} \left(
\frac{\alpha_1 \beta_1}{\alpha_2 \beta_2}  \right)^{s-r-1}
\vartheta^k_{r,s-1}(\beta_1,\beta_2,\alpha_1,\alpha_2)
 \left(\frac{\alpha_1+\alpha_2-\beta_1-\beta_2}{2} \right)^{2 k}
\\
&& \hspace*{-1.2cm} \times \left\{ \frac{t^{(2 k+1) \nu}}{(2k+1)!}
\,{}_{2}\psi_{2} \left[\begin{array}{ccc}
(2k+2,2)\!\! & \!\! (1,1) & \\
&                                                                   & ; \frac{[t^\nu (\alpha_1+\alpha_2+\beta_1+\beta_2)]^2}{4} \\
(1,2)\!\! & \!\!((2 k+1) \nu+1,2 \nu) &
\end{array}\right]
\right.
\\
&& \hspace*{-1.2cm} \left. -\frac{2 t^{2 k \nu}
}{(\alpha_1+\alpha_2+\beta_1+\beta_2) (2k+1)!} \,{}_{2}\psi_{2}
\left[\begin{array}{ccc}
(2k+1,2)\!\! & \!\! (1,1) & \\
&                                                                   & ; \frac{[t^\nu (\alpha_1+\alpha_2+\beta_1+\beta_2)]^2}{4} \\
(0,2)\!\! & \!\!(2 k \nu+1,2 \nu) &
\end{array}\right]
\right\};
\end{eqnarray*}
\begin{eqnarray*}
&& \hspace*{-1.2cm} p^{\nu}_{2r+1,2s+1}(t) =
\sum_{k=|s-r|}^{+\infty} \left( \frac{\alpha_1 \beta_1}{\alpha_2
\beta_2}  \right)^{s-r}
\vartheta^k_{r,s}(\beta_1,\beta_2,\alpha_1,\alpha_2)
 \left(\frac{\alpha_1+\alpha_2-\beta_1-\beta_2}{2} \right)^{2 k}\\
&& \hspace*{-1.2cm} \times \left\{- \frac{t^{2 k \nu}
(\alpha_1+\alpha_2-\beta_1-\beta_2)}{
(\alpha_1+\alpha_2+\beta_1+\beta_2) (2k+1)!} \,{}_{2}\psi_{2}
\left[\begin{array}{ccc}
(2k+1,2)\!\! & \!\! (1,1) & \\
&                                                                   & ; \frac{[t^\nu (\alpha_1+\alpha_2+\beta_1+\beta_2)]^2}{4} \\
(0,2)\!\! & \!\!(2 k \nu+1,2 \nu) &
\end{array}\right]
\right.
\\
&& \hspace*{-1.2cm} \left. +\frac{t^{2 k \nu}}{(2k)!}
\,{}_{2}\psi_{2} \left[\begin{array}{ccc}
(2k+1,2)\!\! &\!\! (1,1) & \\
&                                                                   & ; \frac{[t^\nu (\alpha_1+\alpha_2+\beta_1+\beta_2)]^2}{4} \\
(1,2)\!\! &\!\! (2 k \nu+1,2 \nu) &
\end{array}\right]
\right.
\\
&& \hspace*{-1.2cm} \left. -\frac{t^{(2 k+1) \nu}
(\alpha_1+\alpha_2+\beta_1+\beta_2)}{2 (2k)!} \,{}_{2}\psi_{2}
\left[\begin{array}{ccc}
(2k+2,2)\!\! & \!\! (1,1) & \\
&                                                                   & ; \frac{[t^\nu (\alpha_1+\alpha_2+\beta_1+\beta_2)]^2}{4} \\
(2,2)\!\! & \!\!((2 k+1) \nu+1,2 \nu) &
\end{array}\right]
\right.
\\
&& \hspace*{-1.2cm} \left. +\frac{t^{(2 k+1) \nu}
(\alpha_1+\alpha_2-\beta_1-\beta_2)}{2 (2k+1)!} \,{}_{2}\psi_{2}
\left[\begin{array}{ccc}
(2k+2,2)\!\! & \!\! (1,1) & \\
&                                                                   & ; \frac{[t^\nu (\alpha_1+\alpha_2+\beta_1+\beta_2)]^2}{4} \\
(1,2)\!\! & \!\!((2 k+1) \nu+1,2 \nu) &
\end{array}\right]
\right\}.
\end{eqnarray*}
(ii) Assume that $\alpha_1+\alpha_2=\beta_1+\beta_2$. Then, for
all $s,r\in\mathbb{Z}$, we have the following four cases:
\begin{eqnarray*}
&& \hspace*{-1.2cm} p^{\nu}_{2r,2s}(t)=\sum_{k=|s-r|}^{+\infty}
\left( \frac{\alpha_1 \beta_1}{\alpha_2 \beta_2}  \right)^{s-r}
\eta^k_{r,s}(\alpha_1,\alpha_2,\beta_1,\beta_2)
\\
&& \hspace*{-1.2cm} \times \left\{ \frac{t^{2 k \nu}}{(2k)!}
\,{}_{2}\psi_{2} \left[\begin{array}{ccc}
(2k+1,2)\!\! & \!\! (1,1) & \\
&                                                                   & ; [t^\nu (\alpha_1+\alpha_2)]^2 \\
(1,2)\!\! & \!\!(2 k \nu+1,2 \nu) &
\end{array}\right]
\right.
\\
&& \hspace*{-1.2cm} \left. -(\alpha_1+\alpha_2) \frac{t^{(2 k+1)
\nu}}{(2k)!} \,{}_{2}\psi_{2} \left[\begin{array}{ccc}
(2k+2,2)\!\! &\!\! (1,1) & \\
&                                                                   & ;  [t^\nu (\alpha_1+\alpha_2)]^2 \\
(2,2)\!\! &\!\! ((2 k+1) \nu+1,2 \nu) &
\end{array}\right]
\right\};
\end{eqnarray*}
\begin{eqnarray*}
&& \hspace*{-1.2cm} p^{\nu}_{2r,2s+1}(t) =\alpha_1
\sum_{k=|s-r|}^{+\infty} \left( \frac{\alpha_1 \beta_1}{\alpha_2
\beta_2}  \right)^{s-r}
\eta^k_{r,s}(\alpha_1,\alpha_2,\beta_1,\beta_2)
\\
&& \hspace*{-1.2cm} \times \left\{ \frac{t^{(2 k+1) \nu}} {
(2k+1)!} \,{}_{2}\psi_{2} \left[\begin{array}{ccc}
(2k+2,2)\!\! & \!\! (1,1) & \\
&                                                                   & ;  [t^\nu (\alpha_1+\alpha_2)]^2\\
(1,2)\!\! & \!\!((2 k+1) \nu+1,2 \nu) &
\end{array}\right]
\right.
\\
&& \hspace*{-1.2cm} \left. -\frac{t^{2 k
\nu}}{(\alpha_1+\alpha_2)(2k+1)!} \,{}_{2}\psi_{2}
\left[\begin{array}{ccc}
(2k+1,2)\!\! &\!\! (1,1) & \\
&                                                                   & ;  [t^\nu (\alpha_1+\alpha_2)]^2\\
(0,2)\!\! &\!\! (2 k \nu+1,2 \nu) &
\end{array}\right]
\right\}
\\
&& \hspace*{-1.2cm} +\alpha_2 \sum_{k=|s-r+1|}^{+\infty} \left(
\frac{\alpha_1 \beta_1}{\alpha_2 \beta_2}  \right)^{s-r+1}
\eta^k_{r,s+1}(\alpha_1,\alpha_2,\beta_1,\beta_2)
\\
&& \hspace*{-1.2cm} \times \left\{ \frac{t^{(2 k+1) \nu}}{(2k+1)!}
\,{}_{2}\psi_{2} \left[\begin{array}{ccc}
(2k+2,2)\!\! & \!\! (1,1) & \\
&                                                                   & ;  [t^\nu (\alpha_1+\alpha_2)]^2 \\
(1,2)\!\! & \!\!((2 k+1) \nu+1,2 \nu) &
\end{array}\right]
\right.
\\
&& \hspace*{-1.2cm} \left. -\frac{t^{2 k \nu}
}{(\alpha_1+\alpha_2) (2k+1)!} \,{}_{2}\psi_{2}
\left[\begin{array}{ccc}
(2k+1,2)\!\! & \!\! (1,1) & \\
&                                                                   & ; [t^\nu (\alpha_1+\alpha_2)]^2 \\
(0,2)\!\! & \!\!(2 k \nu+1,2 \nu) &
\end{array}\right]
\right\};
\end{eqnarray*}

\begin{eqnarray*}
&& \hspace*{-1.2cm} p^{\nu}_{2r+1,2s}(t) = \beta_2
\sum_{k=|s-r|}^{+\infty} \left( \frac{\alpha_1 \beta_1}{\alpha_2
\beta_2}  \right)^{s-r}
\eta^k_{r,s}(\beta_1,\beta_2,\alpha_1,\alpha_2)
\\
&& \hspace*{-1.2cm} \times \left\{ \frac{t^{(2 k+1) \nu}} {
(2k+1)!} \,{}_{2}\psi_{2} \left[\begin{array}{ccc}
(2k+2,2)\!\! & \!\! (1,1) & \\
&                                                                   & ;  [t^\nu (\alpha_1+\alpha_2)]^2\\
(1,2)\!\! & \!\!((2 k+1) \nu+1,2 \nu) &
\end{array}\right]
\right.
\\
&& \hspace*{-1.2cm} \left. -\frac{ t^{2 k
\nu}}{(\alpha_1+\alpha_2)(2k+1)!} \,{}_{2}\psi_{2}
\left[\begin{array}{ccc}
(2k+1,2)\!\! &\!\! (1,1) & \\
&                                                                   & ; [t^\nu (\alpha_1+\alpha_2)]^2 \\
(0,2)\!\! &\!\! (2 k \nu+1,2 \nu) &
\end{array}\right]
\right\}
\\
&& \hspace*{-1.2cm} +\beta_1 \sum_{k=|s-r-1|}^{+\infty} \left(
\frac{\alpha_1 \beta_1}{\alpha_2 \beta_2}  \right)^{s-r-1}
\eta^k_{r,s-1}(\beta_1,\beta_2,\alpha_1,\alpha_2)
\\
&& \hspace*{-1.2cm} \times \left\{ \frac{t^{(2 k+1) \nu}}{(2k+1)!}
\,{}_{2}\psi_{2} \left[\begin{array}{ccc}
(2k+2,2)\!\! & \!\! (1,1) & \\
&                                                                   & ; [t^\nu (\alpha_1+\alpha_2)]^2\\
(1,2)\!\! & \!\!((2 k+1) \nu+1,2 \nu) &
\end{array}\right]
\right.
\\
&& \hspace*{-1.2cm} \left. -\frac{t^{2 k \nu}
}{(\alpha_1+\alpha_2) (2k+1)!} \,{}_{2}\psi_{2}
\left[\begin{array}{ccc}
(2k+1,2)\!\! & \!\! (1,1) & \\
&                                                                   & ;  [t^\nu (\alpha_1+\alpha_2)]^2\\
(0,2)\!\! & \!\!(2 k \nu+1,2 \nu) &
\end{array}\right]
\right\};
\end{eqnarray*}
\begin{eqnarray*}
&& \hspace*{-1.2cm}
p^{\nu}_{2r+1,2s+1}(t) =  \sum_{k=|s-r|}^{+\infty} \left( \frac{\alpha_1 \beta_1}{\alpha_2 \beta_2}  \right)^{s-r} \eta^k_{r,s}(\beta_1,\beta_2,\alpha_1,\alpha_2) \\
&& \hspace*{-1.2cm} \times \left\{\frac{t^{2 k \nu}}{ (2k)!}
\,{}_{2}\psi_{2} \left[\begin{array}{ccc}
(2k+1,2)\!\! & \!\! (1,1) & \\
&                                                                   & ;  [t^\nu (\alpha_1+\alpha_2)]^2\\
(1,2)\!\! & \!\!(2 k \nu+1,2 \nu) &
\end{array}\right]
\right.
\\
&& \hspace*{-1.2cm} \left. -\frac{t^{(2 k+1) \nu}
(\alpha_1+\alpha_2)}{(2k)!} \,{}_{2}\psi_{2}
\left[\begin{array}{ccc}
(2k+2,2)\!\! &\!\! (1,1) & \\
&                                                                   & ;  [t^\nu (\alpha_1+\alpha_2)]^2\\
(2,2)\!\! &\!\! ((2 k+1) \nu+1,2 \nu) &
\end{array}\right]
\right\}.
\end{eqnarray*}
\end{proposition}

\begin{remark}\label{rem:how-to-recover-the-nonfractional-case-fractional}
If $\nu=1$, then Proposition \ref{prop:pmf-expressions-fractional}
coincides with Proposition \ref{prop:pmf-expressions} noting that
$${}_{2}\psi_{2} \left[\begin{array}{ccc}
(\zeta_1,2) & (1,1) & \\
&                                                                   & ; z\\
(\omega_1,2) & (\zeta_1,2) &
\end{array}\right]=\left\{\begin{array}{ll}
\sqrt{z} \sinh({\sqrt{z})} &\ \mbox{if}\ \omega_1=0\\
\cosh({\sqrt{z})} &\ \mbox{if}\ \omega_1=1\\
\frac{\sinh({\sqrt{z})}}{\sqrt{z}}&\ \mbox{if}\ \omega_1=2.
\end{array}\right.$$
\end{remark}

We conclude with final proposition and we refer again to the
generalized Fox-Wright function in \eqref{ppsiq}.

\begin{proposition}\label{prop:pmf-expressions-TSS}
Let $\{\tilde{p}_{k,n}^{\nu,\mu}(t):k,n\in\mathbb{Z},t\geq 0\}$ be
as in \eqref{eq:pmf-notation-TSS}.\\
(i) Assume that $\alpha_1+\alpha_2\neq\beta_1+\beta_2$. Then, for
all $s,r\in\mathbb{Z}$, we have the following four cases:
\begin{eqnarray*}
&& \hspace*{-1cm} \tilde p^{\nu,\mu}_{2r,2 s}(t)={\rm e}^{\mu^\nu
t} \sum_{n=|s-r|}^{+\infty} \left(
\frac{\beta_1+\beta_2-\alpha_1-\alpha_2}{\alpha_1+\beta_1+\alpha_2+
\beta_2+2 \mu}  \right)^{2 n} \left( \frac{\alpha_1
\beta_1}{\alpha_2 \beta_2}  \right)^{s-r}
\vartheta^n_{r,s}(\alpha_1,\alpha_2,\beta_1,\beta_2)
\\
&& \hspace*{-1.2cm} \times  \left\{\frac{1}{(2n)!}
\,{}_{1}\psi_{1} \left[\begin{array}{cc}
(1,\nu)\!\!  & \\
                   & ; -t \left(\frac{\alpha_1+\beta_1+\alpha_2+\beta_2}{2}+\mu \right)^{\nu} \\
(1-2n,\nu)\!\! &
\end{array}\right]
\right.
\\
&& \hspace*{-1.2cm}  \left.
-\frac{\beta_1+\beta_2-\alpha_1-\alpha_2}{\alpha_1+\beta_1+\alpha_2+
\beta_2+2 \mu}
\frac{1}{(2n+1)!} \,{}_{1}\psi_{1} \left[\begin{array}{cc}
(1,\nu)\!\!  & \\
                   & ; -t \left(\frac{\alpha_1+\beta_1+\alpha_2+\beta_2}{2}+\mu \right)^{\nu} \\
(-2n,\nu)\!\! &
\end{array}\right]\right\};
\end{eqnarray*}
\begin{eqnarray*}
&& \hspace*{-1.2cm} \tilde p^{\nu,\mu}_{2r,2 s+1}(t)=\frac{2 {\rm
e}^{\mu^\nu t} }{\alpha_1+\alpha_2+\beta_1+\beta_2+2\mu}
\left\{-\alpha_1 \sum_{n=|s-r|}^{+\infty} \left(
\frac{\beta_1+\beta_2-\alpha_1-\alpha_2}{\alpha_1+\beta_1+\alpha_2+
\beta_2+2 \mu}  \right)^{2 n} \left( \frac{\alpha_1
\beta_1}{\alpha_2 \beta_2}  \right)^{s-r} \right.
\\
&& \hspace*{-1.2cm} \left. \times
\vartheta^n_{r,s}(\alpha_1,\alpha_2,\beta_1,\beta_2)
\frac{1}{(2n+1)!} \,{}_{1}\psi_{1} \left[\begin{array}{cc}
(1,\nu)\!\!  & \\
                   & ; -t \left(\frac{\alpha_1+\beta_1+\alpha_2+\beta_2}{2}+\mu \right)^{\nu} \\
(-2n,\nu)\!\! &
\end{array}\right]\right.
\\
&& \hspace*{-1.2cm} \left.
-\alpha_2 \sum_{n=|s-r+1|}^{+\infty} \left(
\frac{\beta_1+\beta_2-\alpha_1-\alpha_2}{\alpha_1+\beta_1+\alpha_2+
\beta_2+2 \mu}  \right)^{2 n}
 \left( \frac{\alpha_1
\beta_1}{\alpha_2 \beta_2}  \right)^{s-r+1}
\vartheta^n_{r,s+1}(\alpha_1,\alpha_2,\beta_1,\beta_2) \right.
\\
&& \hspace*{-1.2cm} \left.
\times \frac{1}{(2n+1)!} \,{}_{1}\psi_{1} \left[\begin{array}{cc}
(1,\nu)\!\!  & \\
                   & ; -t \left(\frac{\alpha_1+\beta_1+\alpha_2+\beta_2}{2}+\mu \right)^{\nu} \\
(-2n,\nu)\!\! &
\end{array}\right]\right\};
\end{eqnarray*}
\begin{eqnarray*}
&& \hspace*{-1.2cm} \tilde p^{\nu,\mu}_{2r+1,2 s}(t)=\frac{2 {\rm
e}^{\mu^\nu t} }{\alpha_1+\alpha_2+\beta_1+\beta_2+2\mu}
\left\{-\beta_2 \sum_{n=|s-r|}^{+\infty} \left(
\frac{\alpha_1+\alpha_2-\beta_1-\beta_2}{\alpha_1+\beta_1+\alpha_2+
\beta_2+2 \mu}  \right)^{2 n} \left( \frac{\alpha_1
\beta_1}{\alpha_2 \beta_2}  \right)^{s-r} \right.
\\
&& \hspace*{-1.2cm} \left. \times
\vartheta^n_{r,s}(\beta_1,\beta_2,\alpha_1,\alpha_2)
\frac{1}{(2n+1)!} \,{}_{1}\psi_{1} \left[\begin{array}{cc}
(1,\nu)\!\!  & \\
                   & ; -t \left(\frac{\alpha_1+\beta_1+\alpha_2+\beta_2}{2}+\mu \right)^{\nu} \\
(-2n,\nu)\!\! &
\end{array}\right]\right.
\\
&& \hspace*{-1.2cm} \left.
 -\beta_1 \sum_{n=|s-r-1|}^{+\infty} \left(
\frac{\alpha_1+\alpha_2-\beta_1-\beta_2}{\alpha_1+\beta_1+\alpha_2+
\beta_2+2 \mu}  \right)^{2 n}  \left( \frac{\alpha_1
\beta_1}{\alpha_2 \beta_2}  \right)^{s-r-1}
\vartheta^n_{r,s-1}(\beta_1,\beta_2,\alpha_1,\alpha_2) \right.
\\
&& \hspace*{-1.2cm} \left. \times
\frac{1}{(2n+1)!} \,{}_{1}\psi_{1} \left[\begin{array}{cc}
(1,\nu)\!\!  & \\
                   & ; -t \left(\frac{\alpha_1+\beta_1+\alpha_2+\beta_2}{2}+\mu \right)^{\nu} \\
(-2n,\nu)\!\! &
\end{array}\right] \right\};
\end{eqnarray*}
\begin{eqnarray*}
&& \hspace*{-1cm} \tilde p^{\nu,\mu}_{2r+1,2 s+1}(t)={\rm
e}^{\mu^\nu t} \sum_{n=|s-r|}^{+\infty} \left(
\frac{\alpha_1+\alpha_2-\beta_1-\beta_2}{\alpha_1+\beta_1+\alpha_2+
\beta_2+2 \mu}  \right)^{2 n} \left( \frac{\alpha_1
\beta_1}{\alpha_2 \beta_2}  \right)^{s-r}
\vartheta^n_{r,s}(\beta_1,\beta_2,\alpha_1,\alpha_2)
\\
&& \hspace*{-1.2cm} \times  \left\{
\frac{1}{(2n)!} \,{}_{1}\psi_{1} \left[\begin{array}{cc}
(1,\nu)\!\!  & \\
                   & ; -t \left(\frac{\alpha_1+\beta_1+\alpha_2+\beta_2}{2}+\mu \right)^{\nu} \\
(1-2n,\nu)\!\! &
\end{array}\right]
-\frac{\alpha_1+\alpha_2-\beta_1-\beta_2}{\alpha_1+\beta_1+\alpha_2+
\beta_2+2 \mu} \right.
\\
&& \hspace*{-1.2cm}  \left. \times
\frac{1}{(2n+1)!} \,{}_{1}\psi_{1} \left[\begin{array}{cc}
(1,\nu)\!\!  & \\
                   & ; -t \left(\frac{\alpha_1+\beta_1+\alpha_2+\beta_2}{2}+\mu \right)^{\nu} \\
(-2n,\nu)\!\! &
\end{array}\right]\right\}.
\end{eqnarray*}
(ii) Assume that $\alpha_1+\alpha_2=\beta_1+\beta_2$. Then, for
all $s,r\in\mathbb{Z}$, we have the following four cases:
\begin{eqnarray*}
&& \hspace*{-2cm}
\tilde p^{\nu,\mu}_{2r,2 s}(t)={\rm e}^{\mu^\nu t}
\sum_{n=|s-r|}^{+\infty} \left( \frac{1}{\alpha_1+\alpha_2+ \mu}
\right)^{2 n} \left( \frac{\alpha_1 \beta_1}{\alpha_2 \beta_2}
\right)^{s-r} \eta^n_{r,s}(\alpha_1,\alpha_2,\beta_1,\beta_2)
\\
&& \hspace*{-0.7cm}
\times
\frac{1}{(2n)!} \,{}_{1}\psi_{1} \left[\begin{array}{cc}
(1,\nu)\!\!  & \\
                   & ; -t \left( \alpha_1+\alpha_2 +\mu \right)^{\nu} \\
(1-2n,\nu)\!\! &
\end{array}\right];
\end{eqnarray*}
\begin{eqnarray*}
&& \hspace*{-1.2cm} \tilde p^{\nu,\mu}_{2r,2 s+1}(t)={\rm
e}^{\mu^\nu t} \left\{-\alpha_1 \sum_{n=|s-r|}^{+\infty} \left(
\frac{1}{\alpha_1+\alpha_2+ \mu}  \right)^{2 n+1} \left(
\frac{\alpha_1 \beta_1}{\alpha_2 \beta_2}  \right)^{s-r}
 \eta^n_{r,s}(\alpha_1,\alpha_2,\beta_1,\beta_2)
 \right.
 \\
&& \hspace*{-1.2cm}
 \left.
\times
\frac{1}{(2n+1)!} \,{}_{1}\psi_{1} \left[\begin{array}{cc}
(1,\nu)\!\!  & \\
                   & ; -t \left( \alpha_1+\alpha_2 +\mu \right)^{\nu} \\
(-2n,\nu)\!\! &
\end{array}\right]
-\alpha_2 \sum_{n=|s-r+1|}^{+\infty} \left(
\frac{1}{\alpha_1+\alpha_2+ \mu}  \right)^{2 n+1}
\right.
 \\
&& \hspace*{-1.2cm} \times
\left.
\left(\frac{\alpha_1 \beta_1}{\alpha_2 \beta_2}  \right)^{s-r+1}
  \eta^n_{r,s+1}(\alpha_1,\alpha_2,\beta_1,\beta_2)
\frac{1}{(2n+1)!} \,{}_{1}\psi_{1} \left[\begin{array}{cc}
(1,\nu)\!\!  & \\
                   & ; -t \left( \alpha_1+\alpha_2 +\mu \right)^{\nu} \\
(-2n,\nu)\!\! &
\end{array}\right]\right\};
\end{eqnarray*}
\begin{eqnarray*}
&& \hspace*{-1.2cm}
\tilde p^{\nu,\mu}_{2r+1,2 s}(t)=
{\rm e}^{\mu^\nu t} \left\{-\beta_2 \sum_{n=|s-r|}^{+\infty}
\left(\frac{1}{\alpha_1+\alpha_2+ \mu}  \right)^{2 n+1}
\left(\frac{\alpha_1 \beta_1}{\alpha_2 \beta_2}  \right)^{s-r}
 \eta^n_{r,s}(\beta_1,\beta_2,\alpha_1,\alpha_2)
 \right.
 \\
&& \hspace*{-1.2cm} \times
 \left.
\frac{1}{(2n+1)!} \,{}_{1}\psi_{1} \left[\begin{array}{cc}
(1,\nu)\!\!  & \\
                   & ; -t \left( \alpha_1+\alpha_2 +\mu \right)^{\nu} \\
(-2n,\nu)\!\! &
\end{array}\right]
-\beta_1 \sum_{n=|s-r-1|}^{+\infty} \left( \frac{1}{\alpha_1+\alpha_2+ \mu}
\right)^{2 n+1}
\right.
\\
&& \hspace*{-1.2cm}  \times
 \left.
\left( \frac{\alpha_1 \beta_1}{\alpha_2 \beta_2}\right)^{s-r-1}
  \eta^n_{r,s-1}(\beta_1,\beta_2,\alpha_1,\alpha_2)
  \frac{1}{(2n+1)!} \,{}_{1}\psi_{1} \left[\begin{array}{cc}
(1,\nu)\!\!  & \\
                   & ; -t \left( \alpha_1+\alpha_2 +\mu \right)^{\nu} \\
(-2n,\nu)\!\! &
\end{array}\right]\right\};
\end{eqnarray*}
\begin{eqnarray*}
&& \hspace*{-1.2cm}
\tilde p^{\nu,\mu}_{2r+1,2 s+1}(t)={\rm e}^{\mu^\nu t}
\sum_{n=|s-r|}^{+\infty} \left( \frac{1}{\alpha_1+\alpha_2+ \mu}
\right)^{2 n} \left( \frac{\alpha_1 \beta_1}{\alpha_2 \beta_2}
\right)^{s-r} \eta^n_{r,s}(\beta_1,\beta_2,\alpha_1,\alpha_2)
\\
&& \hspace*{-1.2cm} \times
  \frac{1}{(2n)!} \,{}_{1}\psi_{1} \left[\begin{array}{cc}
(1,\nu)\!\!  & \\
                   & ; -t \left( \alpha_1+\alpha_2 +\mu \right)^{\nu} \\
(1-2n,\nu)\!\! &
\end{array}\right].
\end{eqnarray*}

\end{proposition}

\begin{remark}\label{rem:how-to-recover-the-nonfractional-case-TSS}
If $\nu=\mu=1$, then Proposition \ref{prop:pmf-expressions-TSS}
coincides with Proposition \ref{prop:pmf-expressions} noting that
$$ \,{}_{1}\psi_{1} \left[\begin{array}{cc}
(1,1)\!\!  & \\
                   & ; -t \left(\frac{\alpha_1+\beta_1+\alpha_2 +\beta_2}{2}+1 \right) \\
(1-l,1)\!\! &
\end{array}\right]
=\frac{(-t)^l}{2^l}
(2+\alpha_1+\alpha_2+\beta_1+\beta_2)^l {\rm e}^{-\frac{t}{2}
(2+\alpha_1+\alpha_2+\beta_1+\beta_2)}.$$
\end{remark}

\end{document}